\pdfoutput=1
\documentclass[11pt,a4paper]{article}
\usepackage[utf8]{inputenc}
\usepackage[a4paper, total={6.6in, 9.6in}]{geometry}

\usepackage[T1]{fontenc}
\usepackage[variablett]{lmodern}
\usepackage{amsmath}
\usepackage{amsfonts}
\usepackage{amssymb}
\usepackage{float}
\usepackage{enumitem}
\usepackage{mathtools}
\usepackage{bbm}

\usepackage{graphicx}
\usepackage{quiver}

\tikzset{every loop/.style={looseness=6}}

\usepackage[bottom]{footmisc}

\usepackage[unicode,naturalnames,hidelinks,hypertexnames=false]{hyperref}
\definecolor{mycolor}{HTML}{750000}
\hypersetup{
    colorlinks = true,
    allcolors = mycolor,
    bookmarksdepth = 2,
}
\usepackage[nameinlink]{cleveref}
\Crefname{subsection}{Subsection}{Subsections}
\Crefname{subsubsection}{Paragraph}{Paragraphs}

\newlist{enumalpha}{enumerate}{1}
\setlist[enumalpha, 1]{label=(\alph*)}
\newlist{enumroman}{enumerate}{1}
\setlist[enumroman, 1]{label=(\roman*)}

\newcommand{\leftl}{\mathopen{}\mathclose\bgroup\left}
\newcommand{\rightr}{\aftergroup\egroup\right}

\usepackage{comment}

\newcommand{\GL}{\mathrm {GL}}
\newcommand{\F}{\mathbb F}
\newcommand{\im}{\mathrm {im} \,}
\renewcommand{\P}{\mathbb P}
\newcommand{\ZZ}{\mathbb Z}
\newcommand{\N}{\mathbb N}
\newcommand{\A}{\mathfrak X}
\newcommand{\Y}{\mathfrak Y}
\newcommand{\Z}{\mathfrak Z}
\newcommand{\T}{\mathfrak T}
\newcommand{\D}{\mathfrak D}
\newcommand{\E}{\mathfrak E}

\newcommand{\Span}{\mathrm {Span}}

\DeclareMathOperator{\Cent}{Cent}
\DeclareMathOperator{\Cl}{Cl}

\newcommand{\Qu}{\mathcal Q}
\newcommand{\Sh}{\mathcal S}
\newcommand{\Oct}{\mathcal O}

\newcommand{\bwp}{\boldsymbol{\wp}}
\newcommand{\Aut}{\mathrm{Aut}}
\newcommand{\Mat}{\mathfrak M}

\newcommand{\Sym}{\mathfrak S}
\newcommand{\Bal}{\mathrm{Bal}}
\newcommand{\JS}{\mathrm{JS}}
\newcommand{\card}[1]{ \left | #1 \right | }
\newcommand{\Isom}{\mathrm{Isom}}
\newcommand{\Hom}{\mathrm{Hom}}
\DeclareMathOperator{\Gr}{Gr}

\DeclareMathOperator{\NP}{NP}

\newcommand{\diag}{\textnormal {diag}}
\newcommand{\eig}{{\textnormal {eig.}/\F_p}}

\renewcommand{\bar}{\overline}

\renewcommand{\tilde}{\widetilde}

\newcommand{\pbinom}[2]{\binom{#1}{#2}_{\!p}}

\newcommand{\customref}[2]{\hyperref[#2]{#1}}
\newcommand\iref[2]{\customref{\Cref*{#1}~\ref*{#2}}{#1}}
\newcommand\iiref[3]{\customref{\Cref*{#1}~\ref*{#2}\ref*{#3}}{#1}}

\usepackage{amsthm}
\usepackage{thmtools}

\newcounter{mycounter}[section]

\theoremstyle{plain}
\newtheorem{theorem}[mycounter]{Theorem}
\newtheorem{corollary}[mycounter]{Corollary}
\newtheorem{proposition}[mycounter]{Proposition}
\newtheorem{lemma}[mycounter]{Lemma}

\theoremstyle{remark}
\newtheorem{remark}[mycounter]{Remark}

\theoremstyle{definition}
\newtheorem{definition}[mycounter]{Definition}

\counterwithin{equation}{section}

\usepackage{titletoc}
\titlecontents{section}
[0pt]                                               
{}%
{\contentsmargin{0pt}                               
    \large\thecontentslabel.\enspace
    }
{\contentsmargin{0pt}\large}                        
{\titlerule*[.5pc]{ }\contentspage}                 
[]                                                  

\usepackage{titlesec}
\titleformat{\section}[block]{\normalfont\centering\scshape\large}{\thesection.}{1em}{}
\titleformat{\subsection}[block]{\normalfont\large}{\thesubsection.}{1em}{\bf}
\titleformat{\subsubsection}[runin]{\normalfont}{\bf\thesubsubsection.}{0.3em}{\bf}

\makeatletter
\patchcmd{\@maketitle}{\LARGE}{\huge}{\typeout{OK 1}}{\typeout{Failed 1}}
\patchcmd{\@maketitle}{\large \lineskip}{\Large \lineskip}{\typeout{OK 2}}{\typeout{Failed 2}}
\makeatother

\title{On matrices commuting with their Frobenius}

\author{
	Fabian Gundlach$^\ast$
	\and%
	Béranger Seguin%
	\footnote{
		Universität Paderborn, Fakultät EIM, Institut für Mathematik, Warburger Str. 100, 33098 Paderborn, Germany.
		Emails:
		\texttt{fabian.gundlach@uni-paderborn.de},
		\texttt{math@beranger-seguin.fr}.
	}
}
\date{}

\renewenvironment{abstract}{%
\par\noindent\rule{\textwidth}{1pt}
\par\noindent\textsc{Abstract.}}
{\par\noindent\rule{\textwidth}{1pt}}

\begin{document}

\maketitle{}

\begin{abstract}
	The \emph{Frobenius} of a matrix~$M$ with coefficients in $\bar\F_p$ is the matrix $\sigma(M)$ obtained by raising each coefficient to the $p$-th power.
	We consider the question of counting matrices with coefficients in~$\F_q$ which commute with their Frobenius, asymptotically when~$q$ is a large power of~$p$.
	We give answers for matrices of size~$2$, for diagonalizable matrices, and for matrices whose eigenspaces are defined over~$\F_p$.
	Moreover, we explain what is needed to solve the case of general matrices.
	We also solve (for both diagonalizable and general matrices) the corresponding problem when one counts matrices $M$ commuting with all the matrices $\sigma(M)$, $\sigma^2(M)$, $\ldots$ in their Frobenius orbit.

	\bigskip

	\par\noindent%
\end{abstract}

{
  \hypersetup{linkcolor=black}
  \tableofcontents{}
}
\par\noindent\rule{\textwidth}{1pt}

\section{Introduction}

Throughout the paper, we fix a prime power $p$ and an integer $n\geq2$.
For any perfect field $K$, we denote by~$\Mat_n(K)$ the ring of $n\times n$-matrices with coefficients in $K$ and by~$\Mat_n^\diag(K)$ the subset of matrices that are diagonalizable over the algebraic closure $\bar K$ (equivalently, these are the matrices which are semi-simple).
We denote by $\sigma$ the Frobenius automorphism of the $\F_p$-algebra $\Mat_n(\bar\F_p)$ acting entrywise by $x \mapsto x^p$.
The symbol $q$ always denotes a power of~$p$.

\subsection{Context}

A classical family of problems is the enumeration of matrices over finite fields satisfying a given property.
For instance, it is shown in~\cite{finenilp,gerstenilp} that there are $q^{n(n-1)}$ nilpotent matrices in $\Mat_n(\F_q)$.
Other examples include \cite{symnilp} (for symmetric nilpotent matrices), \cite{charpol} (for the distribution of characteristic polynomials), \cite{feitfine} (for pairs of commuting matrices), \cite{mutanni} (for mutually annihilating pairs of matrices), etc.

Another classical problem is the study of matrices with coefficients in a differential ring which commute with their coefficientwise derivative, a problem introduced in a letter of Schur in 1934 and which has since been actively studied (see \cite{surveycommdiff} for a survey, and \cite{eigcommdiff} for a characterization of matrices commuting with their derivatives of all orders).

In this paper, we investigate an analogous problem: the enumeration of matrices $A \in \Mat_n(\F_q)$ which commute with their Frobenius $\sigma(A)$,%
\footnote{
	Equivalently, this means that $A$ commutes with $(\sigma-\mathrm{id})A$, making the connection with matrices commuting with their derivative more explicit as $(\sigma-\mathrm{id})$ is a $\sigma$-derivation in the sense that $(\sigma-\mathrm{id})(AB) = (\sigma-\mathrm{id})(A) \, B + \sigma(A)\, (\sigma-\mathrm{id})(B)$.
	The ``constants'' are the matrices with coefficients in $\F_p$.
}
or with their entire Frobenius orbit $\sigma(A), \sigma^2(A), \ldots$
More precisely, we obtain estimates for the number of such matrices.

\subsection{Main results}

For any perfect field~$K$ containing~$\F_p$, define the following subsets of $\Mat_n(K)$:
\begin{align*}
	\A(K)
	& :=
	\Bigl\{
		M\in\Mat_n(K)
		\;\Big\vert\;
		M \textnormal{ and } \sigma(M) \textnormal{ commute}
	\Bigr\},
	\\
	\A^\diag(K)
	& :=
	\A(K) \cap \Mat^\diag_n(K),
	\\
	\A_\infty(K)
	& :=
	\Bigl\{
		M\in\Mat_n(K)
		\;\Big\vert\;
		M, \sigma(M), \sigma^2(M), \ldots \textnormal{ commute pairwise}
	\Bigr\},
	\\
	\A_\infty^\diag(K)
	& :=
	\A_\infty(K) \cap \Mat^\diag_n(K).
\end{align*}
The natural inclusions between these sets are summarized as follows: both $\A^\diag(K)$ and $\A_\infty(K)$ are contained in $\A(K)$ and contain $\A^\diag_\infty(K)$.
When $K$ is omitted from the notation, it is implied that $K = \bar\F_p$, so $\A = \A(\bar\F_p) \subseteq \Mat_n(\bar\F_p)$, etc.
Note that $\A(\F_q) = \A \cap \Mat_n(\F_q)$, etc.

In this paper, we estimate the asymptotic sizes of these sets in the case $K = \F_q$ as $q\to\infty$ ($p$ and $n$ are fixed, and $q$ is a power of~$p$).
Our main result is the following theorem:

\begin{theorem}
	\label{thm:main}
	Letting $\F_q$ be any finite field containing $\F_p$, we have the following estimates, where the implied constants in the $O$-estimates are all independent of $q$:
	\[\arraycolsep=0.5pt\def\arraystretch{1.5}
	\begin{array}{cccclcl}
		\bigl|
			\A^\diag(\F_q)
		\bigr|
		& {}={} &
		c^\diag(p,n)
		& {}\cdot{} &
		q^{\lfloor n^2/3\rfloor + 1}
		& {}+{} &
		O_{p,n}\Bigl(
			q^{\lfloor n^2/3\rfloor + 1/2}
		\Bigr),
		\\
		|\A_\infty^\diag(\F_q)|
		& {}={} &
		p^{n^2-n}
		& {}\cdot{} &
		q^n
		& {}+{} &
		O_{p,n}(q^{n-1}),
		\\
		|\A_\infty(\F_q)|
		& {}={} &
		c_\infty(p,n)
		& {}\cdot{} &
		q^{\lfloor n^2/4 \rfloor+1}
		& {}+{} &
		O_{p,n}\leftl(
			q^{\lfloor n^2/4 \rfloor}
		\rightr),
	\end{array}\]
	where
	\[
		c^\diag(p,n)
		=
		\begin{cases}
			p^2 & \textnormal{if }n=2 \\
			2 & \textnormal{if }n=4 \\
			1 & \textnormal{if }n\notin\{2,4\}
		\end{cases}
	\qquad
	\textnormal{and}
	\qquad
		c_\infty(p,n)
		=
		\begin{cases}
			p^2+p+1 & \textnormal{if }n=2 \\
			\arraycolsep=0pt
			\begin{array}{r}
				p^6 + p^5 + 3 p^4 + 3 p^3 \\
				+ 3 p^2 + p + 1
			\end{array}
			& \textnormal{if } n=3 \\
			\pbinom{n}{n/2} &
			\textnormal{if $n\geq4$ is even} \\
			2\pbinom{n}{\lfloor n/2 \rfloor} &
			\textnormal{if $n\geq5$ is odd.}
		\end{cases}
	\]
	(The Gaussian binomial coefficient $\pbinom{n}{k}$ is the number of $k$-dimensional subspaces of~$\F_p^n$.)
\end{theorem}

The estimates for
$\bigl|
	\A^\diag(\F_q)
\bigr|$
are proved in \Cref{thm:1-diag}, and they are certainly the most technical and interesting part of this article.
The estimates for $\bigl|\A_\infty^\diag(\F_q)\bigr|$ are proved in \Cref{thm:infty-diag}.
The estimates for $\bigl|\A_\infty(\F_q)\bigr|$ are proved in \Cref{thm:infty-all} (for $n \geq 3$) and \Cref{cor:count-2-by-2} (for $n=2$).

A consequence of \Cref{thm:main} is that, for any $n \geq 3$ and for any sufficiently large $q$, since $\lfloor n^2/3\rfloor > \lfloor n^2/4 \rfloor$, there are significantly more matrices $M$ commuting with $\sigma(M)$ (even among diagonalizable matrices) than there are matrices commuting with the entire orbit $\sigma(M),\sigma^2(M),\dots$

We do not obtain general estimates for $\bigl|\A(\F_q)\bigr|$.
However, we relate the exponent of $q$ in the asymptotics of $|\A(\F_q)|$ (as $q\to\infty$) to the dimensions of intersections $\Cent M \cap \Cl M$, where~$\Cent M$ and $\Cl M$ respectively denote the centralizer and conjugacy class of a matrix $M\in\Mat_n(\bar\F_p)$.
More precisely, define for any $M\in\Mat_n(\bar\F_p)$ the integer
\begin{equation}
	\label{eq:def-dM}
	d(M)
	:=
	(\textnormal{number of distinct eigenvalues of~$M$})
	+
	\dim\bigl({
		\Cent M \cap \Cl M
	}\bigr).
\end{equation}
We prove a general statement (\Cref{prop:counting-by-shape}), which implies the following:
\begin{theorem}
	\label{thm:1-all}
	For any finite field $\F_q \supseteq \F_p$, we have $|\A(\F_q)| = |\A^\diag(\F_q)| + O_{p,n}(q^{a_{p,n}})$, where~$a_{p,n}$ is the maximum value of~$d(M)$ over non-diagonalizable matrices $M\in\Mat_n(\bar\F_p) \setminus \Mat_n^\diag(\bar\F_p)$.
\end{theorem}

We are unable to compute $d(M)$ in general.
This is related to the hard problem of classifying pairs of commuting matrices up to simultaneous conjugation.
In \Cref{subsn:eigenspaces-def-over-fp}, we deal with a special case where that problem becomes tractable, in order to illustrate how the principle behind \Cref{prop:counting-by-shape} may be applied.
Specifically, we prove the following theorem about the sets $\A^\eig(\F_q)$ of matrices $M \in \A(\F_q)$ whose eigenspaces are all defined over $\F_p$:

\begin{theorem}[cf.\ \Cref{thm:eig}]
	\label{thm:eig-intro}
	For any finite field $\F_q\supseteq\F_p$, we have
	\[
		\bigl|
			\A^\eig(\F_q)
		\bigr|
		=
		c^\eig(p,n)\cdot q^{\lfloor n^2/4\rfloor+1}
		+ O_{p,n}(q^{\lfloor n^2/4\rfloor}),
	\]
	for specific constants $c^\eig(p,n)$, given in \Cref{thm:eig}.
\end{theorem}

\subsection{Outline and strategy}

In \Cref{sn:diag}, we prove \Cref{thm:1-diag} about~$\A^\diag$.
Using the Lang--Weil bound, the claim reduces to the computation of geometric invariants of the constructible subset $\A^\diag \subseteq \Mat_n(\bar\F_p)$, namely its dimension and the number of its irreducible components of maximal dimension.
To determine the top-dimensional irreducible components of $\A^\diag$, we stratify this set according to how the eigenspaces of a matrix~$M$ intersect those of~$\sigma(M)$, using quivers to encode this combinatorial information.

In \Cref{sn:general}, we show \Cref{prop:counting-by-shape} (and thus \Cref{thm:1-all}).
To relate the dimension of $\A$ to the numbers $d(M)$ defined above, we stratify $\Mat_n(\F_q)$ according to the shape of the Jordan normal form of matrices, i.e., the number of Jordan blocks of each size for each eigenvalue.

In \Cref{subsn:eigenspaces-def-over-fp}, we prove \Cref{thm:eig}, counting matrices in $\A(\F_q)$ with eigenspaces defined over~$\F_p$.
This special case lets us illustrate the principle described in \Cref{sn:general}, and is made accessible by the fact that classifying pairs of commuting matrices \emph{whose eigenspaces coincide} is relatively easy (cf.~\Cref{prop:compute-dim-esh}/\Cref{lem:dim-cent-sameker}).

In \Cref{sn:stable}, we prove \Cref{thm:infty-diag,thm:infty-all}.
First, we quickly deal with the special case $n=2$ (\Cref{cor:count-2-by-2}).
Then, we study~$\A^\diag_\infty$ and~$\A_\infty$, and we observe (see~\Cref{lemma:stable}) that for any matrix $M \in \A_\infty$, its Frobenius orbit $\bigl(\sigma^i(M)\bigr)_{i \geq 0}$ generates a commutative algebra of $\Mat_n(\bar\F_p)$ defined over~$\F_p$---moreover consisting of simultaneously diagonalizable matrices if $M \in \A^\diag_\infty$.
Hence, the statements boil down to counting specific subalgebras of~$\Mat_n(\F_p)$: the case of $\A_\infty$ essentially boils down to a classical result of Schur (cf.~\cite{schur,mirza}), and the case of $\A_\infty^\diag(\F_q)$ is a special case of a result of Steinberg (cf.~\cite[14.16]{steinberg}/\cite[Theorem~5.8]{repstab}).

\subsection{Motivation and related results}

Our initial contact with this problem came from the role played by analogous counts in the distribution of wildly ramified extensions of the local function field $\F_q(\!(T)\!)$, see \cite[Propositions~4.6 and~4.9]{wildcount}.
In \cite[Lemmas~6.3, 6.4, 6.5]{wildcount}, we estimate the number of matrices commuting with their Frobenius (as well as with the Frobenius of their Frobenius, etc.) in a specific group of invertible matrices, the Heisenberg group~$H_k(\F_q)$.
This has allowed us to describe the distribution of $H_k(\F_p)$-extensions of $\F_q(\!(T)\!)$ and $\F_q(T)$.
We were led to generalize that question to more general matrices, and to study it for itself, after realizing that it was a deep and non-trivial problem.

A different point of view is that we are counting the $(\F_q, \sigma)$-points of the difference scheme defined by the difference equation $M \sigma(M) = \sigma(M) M$ (for~$\A$).
This makes our problem fit into the general framework of Hrushovski--Lang--Weil estimates, cf. \cite{langweiltordu,hils2024langweiltypeestimatesfinite}.
Alternatively, one can define the variety of pairs of commuting matrices (an irreducible subvariety of~$\mathbb A^{2n^2}_{\F_p}$ which is well-studied, see e.g.~\cite{motzkintaussky,gersten2,guralnick,commpair})
and describe the geometry (dimension, irreducible components,~...) of its intersection with the graph of~$\sigma$ (also a subvariety of~$\mathbb A^{2n^2}_{\F_p}$).

\subsection{Terminology and conventions}

A linear subspace~$V \subseteq \bar\F_p^n$ is \emph{defined over $\F_{p^r}$} if it is $\sigma^r$-invariant, i.e., $\sigma^r(V)=V$.
By Galois descent for vector spaces, this is equivalent to the vector space having a basis consisting of vectors in $\F_{p^r}^n$, i.e., to the existence of an isomorphism $V\simeq V'\otimes_{\F_{p^r}}\bar\F_p$ for the $\F_{p^r}$-vector space $V' = V\cap\F_{p^r}^n$.

In this paper, the word \emph{variety} always refers to a classical quasi-projective variety over~$\bar\F_p$, i.e., a (Zariski) locally closed subset of~$\P^r(\bar\F_p)$ for some~$r \geq 1$.
We do \emph{not} assume that varieties are irreducible.
We say that a variety $V \subseteq \P^r(\bar\F_p)$ is \emph{defined over $\F_p$} if it is $\sigma$-invariant, i.e., $\sigma(V) = V$ where $\sigma \colon \P^r(\bar\F_p) \to \P^r(\bar\F_p)$ is induced by $\sigma \colon x \mapsto x^p$.
If $V \subseteq \P^r(\bar\F_p)$ is a variety and $q$ is a power of $p$, we let $V(\F_q) := V \cap \P^r(\F_q)$.
The \emph{dimension} of a constructible subset of~$\P^r(\bar\F_p)$ is the (Krull) dimension of its Zariski closure.
A regular map $f \colon X \to Y$ between smooth varieties is \emph{étale} if for every $x\in X$, the differential $D_xf \colon T_xX\to T_{f(x)}Y$ of~$f$ at~$x$ is an isomorphism of $\bar\F_p$-vector spaces.%
\footnote{
	By Hilbert's Nullstellensatz, varieties form a category equivalent to that of reduced quasi-projective schemes over~$\bar\F_p$.
	A variety $V$ is defined over $\F_p$ if and only if the corresponding reduced $\bar\F_p$-subscheme of $\P^r_{\bar\F_p}$ is obtained via extension of scalars of a geometrically reduced $\F_p$-subscheme of $\P^r_{\F_p}$.
	In that case, the elements of $V(\F_q)$ actually correspond to the $\F_q$-points of $V$.
	A regular map between smooth varieties is étale if and only if the corresponding morphism of reduced smooth quasi-projective schemes is étale.
}

\subsection{Acknowledgments}

This work was supported by the Deutsche Forschungsgemeinschaft (DFG, German Research Foundation) --- Project-ID 491392403 --- TRR 358 (Project A4).
The authors thank Jürgen Klüners for helpful feedback on a first version.

\section{Diagonalizable matrices commuting with their Frobenius}
\label{sn:diag}

In this section, we determine the asymptotics of $|\A^\diag(\F_q)|$, i.e., we prove \Cref{thm:1-diag}.
In \Cref{subsn:matrices-and-quivers}, we associate to any such matrix a quiver $\Qu$ encoding the dimensions of the intersections of the eigenspaces of~$M$ with those of~$\sigma(M)$.
This will let us write $\A^\diag$ as a disjoint union of equidimensional constructible subsets $\A^\diag_\Qu \subseteq \Mat_n(\bar\F_p)\simeq\bar\F_p^{n^2}$.
In \Cref{subsn:octopus}, we identify those quivers $\Qu$ for which the dimension of $\A^\diag_\Qu$ is maximal, and in \Cref{subsn:diag-2,subsn:tool-irred,subsn:octopus-variety,subsn:diag-dumbbell}, we compute the irreducible components of the corresponding sets $\A^\diag_\Qu$, and we show that they are defined over~$\F_p$.
This allows us to prove \Cref{thm:1-diag} using the Lang--Weil bound in \Cref{subsn:diag-conclusion}.

\subsection{Diagonalizable matrices and their associated quivers}
\label{subsn:matrices-and-quivers}

\paragraph{Balanced quivers.}
A \emph{quiver} is a finite directed graph in which one also allows loops (from a vertex to itself) and multiple parallel edges.
We say that a vertex of a quiver is \emph{isolated} if there are no edges (including loops) having that vertex as either source or target.
We say that a quiver is \emph{balanced} if, for each vertex, equally many edges have that vertex as source and as target (i.e., in-degrees and out-degrees coincide).
If $\Qu$ is a quiver, we denote by $V(\Qu)$ the set of its vertices, and by~$\Qu(i,j)$ the set of edges $i \to j$ for any $i,j \in V(\Qu)$.
Assuming that $\Qu$ is balanced, we also define the degree $d_{\Qu}(i) := \sum_{j \in V(\Qu)} |\Qu(i,j)| = \sum_{j \in V(\Qu)} |\Qu(j,i)|$ of each vertex $i \in V(\Qu)$.
We let $\Bal_n$ be the (finite) set of isomorphism classes of balanced quivers with no isolated vertices and $n$ edges.

\paragraph{Quiver associated to a matrix.}
Let $M \in \Mat_n(\bar\F_p)$.
For each eigenvalue $\lambda$ of~$M$, let $E_\lambda$ be the eigenspace $\ker(M - \lambda I_n)$.
Note that $\sigma(E_\lambda) = \ker\bigl( \sigma(M) - \sigma(\lambda) I_n \bigr)$ is the eigenspace of~$\sigma(M)$ for the eigenvalue $\sigma(\lambda)$.

\begin{definition}
	We associate to the matrix $M$ a quiver $\Qu_M$ defined as follows:
	\begin{itemize}
		\item
			its vertices are the eigenvalues $\lambda$ of~$M$;
		\item
			for any eigenvalues  $\lambda, \mu$, the number of edges $\lambda \to \mu$ is the dimension of~$E_\lambda \cap \sigma(E_\mu)$.
	\end{itemize}
\end{definition}

\begin{proposition}
	\label{prop:quM-in-Baln}
	Let $M\in\Mat_n^\diag(\bar\F_p)$.
	Then, $M\in\A^\diag$ if and only if the corresponding quiver~$\Qu_M$ has exactly $n$ edges.
	In that case, $\Qu_M \in \Bal_n$, and $\dim E_\lambda = d_{\Qu_M}(\lambda)$ for all eigenvalues $\lambda$.
\end{proposition}

\begin{proof}
	Since $\bigoplus_\lambda E_\lambda = \bar\F_p^n$ and $\bigoplus_\lambda\sigma(E_\lambda)=\bar\F_p^n$, the spaces $E_\lambda\cap\sigma(E_\mu)$ are always linearly independent.
	The diagonalizable matrices $M$ and $\sigma(M)$ commute if and only if they are simultaneously diagonalizable, i.e., if and only if
	\[
		\bigoplus_{\lambda,\mu}
			\Bigl({
				E_\lambda
				\cap
				\sigma(E_\mu)
			}\Bigr)
		= \bar\F_p^n,
	\]
	meaning that the quiver $\mathcal Q_M$ has exactly $n$ edges.
	In that case, for any eigenvalue $\lambda$ of~$M$, we have
	\[
		\bigoplus_\mu
			\Bigl({
				E_\lambda
				\cap
				\sigma(E_\mu)
			}\Bigr)
		=
		E_\lambda
		\simeq
		\sigma(E_\lambda)
		=
		\bigoplus_\mu
			\Bigl({
				E_\mu
				\cap
				\sigma(E_\lambda)
			}\Bigr),
	\]
	so the quiver is balanced and satisfies $d_{\Qu_M}(\lambda)=\dim E_\lambda$ (in particular, it has no isolated vertices).
\end{proof}

\paragraph{The space of matrices having a given quiver.}

For any quiver $\Qu \in \Bal_n$, we define the subset $\A^\diag_\Qu \subseteq \A^\diag$ of matrices~$M$ such that~$\Qu_M$ and~$\Qu$ are isomorphic.%
\footnote{
	Be aware that this is \emph{not} a quiver variety or a quiver Grassmannian in the usual sense.
}
(The vertices of the quiver~$\Qu$ are purely combinatorial: the $\card {V(\Qu)}$ eigenvalues of a matrix~$M \in \A_\Qu^\diag$ are subject to no constraints other than distinctness.)
\Cref{prop:quM-in-Baln} directly implies:
\begin{equation}
	\label{eqn:partition-diag}
	\A^\diag
	=
	\bigsqcup_{\Qu \in \Bal_n}
		\A^\diag_\Qu.
\end{equation}
We will show that each set $\A^\diag_\Qu$ is constructible, so that, by the Lang--Weil estimates (cf.~\cite{lang-weil}), the leading term in the asymptotics of $|\A^\diag (\F_q)|$ depends on the maximal dimension of~$\A^\diag_\Qu$ over quivers $\Qu \in \Bal_n$, and on the number of irreducible components having that dimension that are defined over $\F_p$.

Fix a quiver $\Qu \in \Bal_n$.
In order to compute the geometric invariants of~$\A^\diag_\Qu$, we explain how to construct all the diagonalizable matrices~$M$ such that $\Qu_M \simeq \Qu$.
For each vertex $i$ of~$\Qu$, we must pick an eigenvalue $\lambda_i$ and an eigenspace $V_i$, making sure that:
\begin{itemize}
	\item
		the eigenvalues $\lambda_i$ are distinct;
	\item
		the eigenspaces $V_i$ are in direct sum, and together span the entire ($n$-dimensional) space;
	\item
		the dimension of~$V_i\cap\sigma(V_j)$ equals the number of edges $i\to j$ in $\Qu$.
\end{itemize}
For any finite-dimensional vector space $V$ and any $k$, we denote by $\Gr_k(V)$ the Grassmannian variety parametrizing $k$-dimensional subspaces of~$V$.
This space has dimension $k(\dim V-k)$ if $0\leq k\leq\dim V$ and is otherwise empty.
(See for example \cite[Lecture~6]{harris-algebraic-geometry} for an introduction to Grassmannians.)
We also write $\P(V) := \Gr_1(V)$ for the projective space parametrizing one-dimensional subspaces of~$V$.
We will repeatedly make use of the fact that for any $k,l,m$, the subset
\begin{equation}\label{eq:fixed-intersection-space}
	\{(A,B)\in\Gr_k(V)\times\Gr_l(V) \mid \dim(A+B) = m\}
\end{equation}
of $\Gr_k(V)\times\Gr_l(V)$ is locally closed, and that the maps defined on that set mapping $(A,B)$ to $A + B \in \Gr_m(V)$ (resp.\ to $A \cap B \in \Gr_{k+l-m}(V)$) are regular.
Moreover, for any $n,k\geq0$, the following map is also regular:
\[
	\Gr_k(\bar\F_p^n) \to \Gr_k(\bar\F_p^n),
	\qquad
	A \mapsto \sigma(A).
\]

Let $r = |V(\Qu)|$, say $V(\Qu)=\{1,\dots,r\}$.
We define the following two quasi-projective varieties:
\begin{itemize}
	\item
		$\Y_{\Qu}$ is the variety of ordered tuples $(\lambda_1,\dots,\lambda_r)$ of distinct elements of~$\bar\F_p$.
		It is a non-empty Zariski open subset of~$\bar\F_p^r$, hence it is Zariski dense and its dimension is $r = |V(\Qu)|$.
	\item
		$\Z_{\Qu}$ is the (locally closed) subspace of~$\Gr_{d_\Qu(1)}(\bar\F_p^n)\times\cdots\times\Gr_{d_\Qu(r)}(\bar\F_p^n)$ consisting of those tuples~$(V_1,\dots,V_r)$ of subspaces of~$\bar\F_p^n$ of dimensions~$d_\Qu(1),\dots,d_\Qu(r)$ which together span $\bar\F_p^n$ and such that $\dim\left(V_i\cap\sigma(V_j)\right) = |\Qu(i,j)|$ for all $i,j$.
\end{itemize}

Sending a pair $\bigl((\lambda_1,\dots,\lambda_r),(V_1,\dots,V_r)\bigr) \in \Y_\Qu \times \Z_\Qu$ to the diagonalizable matrix $M$ with eigenvalues $\lambda_1,\dots,\lambda_r$ and corresponding eigenspaces $V_1,\dots,V_r$, we obtain a regular map
\begin{equation}
	\label{eqn:surject-to-xqu}
	\Y_{\Qu} \times \Z_{\Qu} \to \Mat_n(\bar \F_p)
\end{equation}
whose image is exactly $\A^\diag_\Qu$ by \Cref{prop:quM-in-Baln}.
In particular, $\A^\diag_\Qu$ is a constructible
subset of~$\Mat_n(\bar\F_p)$ by Chevalley's theorem.
The group $\Aut(\Qu)$ consisting of automorphisms of the quiver, i.e., of permutations of the vertices which preserve edge multiplicities, acts simply transitively on each fiber above a point of~$\A^\diag_\Qu$.
Moreover, the Frobenius automorphism acts on the sets $\A^\diag_\Qu$, $\Y_\Qu$, $\Z_\Qu$, and the map from \Cref{eqn:surject-to-xqu} is $\sigma$-equivariant.

To compute the dimension of~$\Z_\Qu$, we use the following lemma:

\begin{lemma}
	\label{lem:wp-etale}
	Let $r\geq1$.
	The map $\bwp \colon \GL_n(\bar\F_p)\to\GL_n(\bar\F_p)$ given by $\bwp(E):= E^{-1}\sigma^r(E)$ is étale and surjective.
	Moreover, $\GL_n(\F_{p^r})$ acts simply transitively on each fiber by left multiplication.
\end{lemma}

Consequently, if~$S$ is a locally closed subset of $\GL_n(\bar\F_p)$ of pure dimension~$d$, then so is~$\bwp^{-1}(S)$.
This fact will be used several times in the paper.

\begin{proof}
	More generally, for any $A\in\GL_n(\bar\F_p)$, consider the map $\bwp_A \colon \GL_n(\bar\F_p)\to\GL_n(\bar\F_p)$ given by $\bwp_A(E):= E^{-1}A\sigma^r(E)$.
	As $p=0$ in $\bar\F_p$, the differential of~$\sigma^r$ is the zero map (at any point); by the product rule, the differential of~$\bwp_A$ at a matrix $E \in \GL_n(\bar\F_p)$ thus maps a tangent vector $\textrm{d}E \in \Mat_n(\bar\F_p)$ to $-E^{-1}\textrm{d} E E^{-1} A \sigma^r(E) \in \Mat_n(\bar\F_p)$.
	Hence, the differential of $\wp_A$ at every point $E$ is a linear isomorphism, so~$\bwp_A$ is étale.
	Since domain and target have the same dimension and $\GL_n(\bar\F_p)$ is irreducible, this implies that $\bwp_A$ is dominant for all $A$.
	The image of $\bwp_A$ (which is dense, and constructible by Chevalley's theorem) then contains a non-empty open subset of $\Mat_n(\bar\F_p)$, hence intersects the (dense) image of~$\bwp_{I_n} = \bwp$.
	We have an equality $\bwp_A(E_1) = \bwp(E_2)$, implying that $A = \bwp(E_2 E_1^{-1})$.
	We have shown that the map $\bwp$ is surjective.
	
	Finally, we have $E^{-1}\sigma^r(E)=E'^{-1}\sigma^r(E')$ if and only if $E'E^{-1} \in \GL_n(\F_{p^r})$, so all non-empty fibers of~$\bwp$ are right $\GL_n(\F_{p^r})$-cosets.
\end{proof}

\begin{lemma}
	\label{lem:dim-Z}
	~
	\begin{enumalpha}
	\item
		\label{item:dim-Z-normal}
		The space $\Z_{\Qu}$ is non-empty and has pure dimension $\sum_i d_{\Qu}(i)^2 - \sum_{i,j} |\Qu(i,j)|^2$, and the finite group $\GL_n(\F_p)$ acts transitively on the set of its irreducible components.
	\item
		\label{item:dim-Z-special}
		Let $k\in V(\Qu)$ with $0<|\Qu(k,k)|<d_\Qu(k)$.
		Consider the locally closed subset $\Z_{\Qu,k} \subseteq \Z_\Qu$ consisting of those tuples $(V_1,\dots,V_r)\in\Z_\Qu$ for which $V_k\cap\sigma(V_k)$ is defined over~$\F_p$.
		This subset has strictly smaller dimension than $\Z_\Qu$.
	\end{enumalpha}
\end{lemma}

\begin{proof}
	~
	\begin{enumalpha}
	\item
		The formulas $U_{ij} := V_i\cap\sigma(V_j)$ and $V_i:=\bigoplus_j U_{ij}$ define two inverse regular maps, showing that~$\Z_{\Qu}$ is isomorphic to the subvariety~$\tilde\Z_\Qu$ of~$\prod_{i,j} \Gr_{|\Qu(i,j)|}(\bar\F_p^n)$ parametrizing tuples $(U_{ij})_{i,j\in[r]}$ of subspaces of~$\bar\F_p^n$ satisfying the following three conditions: $\dim U_{ij} = |\Qu(i,j)|$ for all $i,j \in V(\Qu)$, $\bigoplus_{i,j}U_{ij} = \bar\F_p^n$, and $\sigma(\bigoplus_j U_{ij}) = \bigoplus_j U_{ji}$ for all $i \in V(\Qu)$.

		Define the $\bar\F_p$-vector spaces $C_{ij} := \bar\F_p^{|\Qu(i,j)|}$ and $C := \bigoplus_{i,j} C_{ij}$.
		By definition, $C$ is isomorphic to $\bar\F_p^n$.
		In order to parametrize tuples $(U_{ij})_{i,j}\in\tilde\Z_\Qu$, we consider the surjective regular map
		\[
			f \colon
			\Isom(C, \bar\F_p^n) \to
			\Bigl\{
				(U_{ij})_{i,j}
				\;\Big\vert\;
				\dim U_{ij} = |\Qu(i,j)|
				\textnormal{ and }
				{\textstyle\bigoplus_{i,j} U_{ij}} = \bar\F_p^n
			\Bigr\},
			\quad
			E \mapsto \bigl(E(C_{ij})\bigr)_{i,j},
		\]
		whose fibers are isomorphic to the variety
		\[
			F:= \prod_{i,j}\GL(C_{ij}),
			\;\;\textnormal{of dimension }
			\sum_{i,j} (\dim C_{ij})^2 = \sum_{i,j} |\Qu(i,j)|^2.
		\]
		For any $E\in\Isom(C,\bar\F_p^n)$, let $\sigma(E)$ be the $\bar\F_p$-linear isomorphism obtained as the composition $C \stackrel{\sigma^{-1}}\longrightarrow C \stackrel{E}\to \bar\F_p^n \stackrel{\sigma}\longrightarrow \bar\F_p^n$, where~$\sigma$ acts on $C$ and on $\bar\F_p^n$ in the natural way.
		We have $\sigma(\bigoplus_j U_{ij}) = \bigoplus_j U_{ji}$ if and only if $\bwp(E) := E^{-1}\sigma(E)$ sends $\bigoplus_j C_{ij}$ to $\bigoplus_j C_{ji}$, i.e., if and only if $\bwp(E)$ lies in the irreducible variety
		\[
			S
			:=
			\prod_i
			\Isom\biggl({
				\bigoplus_j C_{ij}, \,
				\bigoplus_j C_{ji}
			}\biggr),
			\;\;\textnormal{of dimension }
			\sum_i
				\Bigl(
					\sum_j
						\dim C_{ij}
				\Bigr)
				\Bigl(
					\sum_j
						\dim C_{ji}
				\Bigr)
			=
			\sum_i
				d_\Qu(i)^2.
		\]
		In other words, $\tilde\Z_\Qu = f(\bwp^{-1}(S))$.
		Together with \Cref{lem:wp-etale}, this implies that $\Z_\Qu \simeq \tilde\Z_\Qu = f(\bwp^{-1}(S))$ is non-empty and has pure dimension
		\[
			\dim \Z_\Qu =
			\dim \bwp^{-1}(S) - \dim F
			= \dim S - \dim F
			= \sum_i d_\Qu(i)^2 - \sum_{i,j} |\Qu(i,j)|^2
		\]
		and that $\GL_n(\F_p)$ acts transitively on the set of its irreducible components.
	\item
		We reason as in~\ref{item:dim-Z-normal}.
		In terms of the notation above, the condition $\sigma(U_{kk})=U_{kk}$ means that~$\bwp(E)$ must send~$C_{kk}$ to itself, so~$S$ must be replaced by the subset $S' := \{ A\in S \mid A(C_{kk}) = C_{kk} \}$, and the claim reduces to showing that $\dim S' < \dim S$.
		We can describe $S$ as the subset of the vector space $\Hom(C_{kk},C_{kk})\times\prod_{(i,j)\neq(k,k)}\Hom(C_{ij},\bigoplus_{j'}C_{j'i}) \subseteq \Hom(C,C)$ formed of those endomorphisms which are invertible, so $S'$ has the same dimension as that vector space, namely
		\begin{align*}
			\dim S'
			&=
			|\Qu(k,k)|^2 + \sum_{(i,j)\neq(k,k)}|\Qu(i,j)|\cdot d_\Qu(i)
			\\
			&=
			|\Qu(k,k)|^2 - |\Qu(k,k)|\cdot d_\Qu(k) + \sum_{i,j}|\Qu(i,j)|\cdot d_\Qu(i)
			\\
			&=
			-|\Qu(k,k)|\cdot(d_\Qu(k)-|\Qu(k,k)|) + \sum_i d_\Qu(i)^2
			\\
			&
			< \sum_i d_\Qu(i)^2
			= \dim S.
			\qedhere
		\end{align*}
	\end{enumalpha}
\end{proof}

\begin{corollary}
	\label{prop:dim-xqu}
	The subset $\A^\diag_\Qu \subseteq \Mat_n(\bar\F_p)$ is constructible, of pure dimension
	\[
		\dim \A^\diag_\Qu
		=
		|V(\Qu)|
		+
		\sum_{i \in V(\Qu)}
			d_{\Qu}(i)^2
		-
		\sum_{i, j \in V(\Qu)}
			|\Qu(i,j)|^2.
	\]
\end{corollary}

\begin{proof}
	Since every fiber of the surjection $\Y_{\Qu}\times \Z_{\Qu}\twoheadrightarrow \A^\diag_\Qu$ is finite (of size~$|\Aut(\Qu)|$), $\A^\diag_\Qu$ is equidimensional and
	\[
		\dim \A^\diag_\Qu
		=
		\dim \Y_\Qu + \dim \Z_\Qu
		\;\underset{\textnormal{\customref{Lem.~\ref*{lem:dim-Z}\ref*{item:dim-Z-normal}}{item:dim-Z-normal}}}{=}\;
		|V(\Qu)| + \sum_i d_{\Qu}(i)^2 - \sum_{i,j}|\Qu(i,j)|^2.
		\qedhere
	\]
\end{proof}

\subsection{The octopus has maximal dimension}
\label{subsn:octopus}

\Cref{prop:dim-xqu} and \Cref{eqn:partition-diag} imply that the dimension of~$\A^\diag$ is the maximal dimension of~$\A^\diag_\Qu$ over quivers $\Qu \in \Bal_n$, and give an explicit formula for the dimension of $\A^\diag_\Qu$ in terms of the quiver~$\Qu$.
We shall now compute this maximal dimension and describe the corresponding optimal quivers.

\begin{proposition}
	\label{prop:octopus}
	Let $n \geq 1$, and let $[n/3]$ be the (uniquely defined) integer closest to $n/3$.
	Then:
	\[
		\max_{\Qu \in \Bal_n}
			\dim \A^\diag_\Qu
		=
		\left\lfloor \frac{n^2}3 \right\rfloor + 1.
	\]
	The maximum is reached by the following quiver with $[n/3]+1$ vertices, which we call the \emph{octopus quiver} (with $n$ edges) and denote by $\Oct_n$:
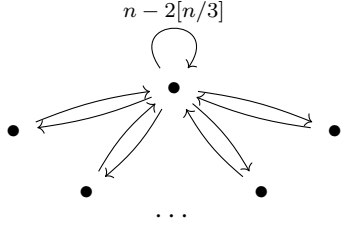
	\[\begin{tikzpicture}
		\node[shape=circle] (O) at (0,0) {$\bullet$};
		\node[shape=circle] (A) at (-15:2.2) {$\bullet$};
		\node[shape=circle] (B) at (-50:1.8) {$\bullet$};
		\node[shape=circle] (C) at (-130:1.8) {$\bullet$};
		\node[shape=circle] (D) at (-165:2.2) {$\bullet$};
		\node at (-90:1.7) {$\cdots$};
		\draw[->,bend left=7] (O) edge (A);
		\draw[<-,bend right=7] (O) edge (A);
		\draw[->,bend left=10] (O) edge (B);
		\draw[<-,bend right=10] (O) edge (B);
		\draw[->,bend left=10] (O) edge (C);
		\draw[<-,bend right=10] (O) edge (C);
		\draw[->,bend left=7] (O) edge (D);
		\draw[<-,bend right=7] (O) edge (D);
		\draw[->] (O) edge[in=55,out=125,loop] node[above,inner sep=0.4ex] {\scriptsize$n-2[n/3]$} (O);
	\end{tikzpicture}\]
	where the number on the top loop means that there are $n - 2 [n/3]$ parallel loops from the central vertex to itself.
	Moreover, up to isomorphism:
	\begin{itemize}
		\item
			When $n \not\in \{2,4\}$, there are no other quivers in $\Bal_n$ maximizing $\dim \A^\diag_\Qu$;
		\item
			When $n=2$, there is a single additional optimal (non-connected) quiver, namely $\Oct_1 \sqcup \Oct_1$:
			
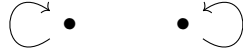
\[\begin{tikzpicture}
			\node[shape=circle] (A) at (0,0) {$\bullet$};
			\node[shape=circle] (B) at (1.5,0) {$\bullet$};
			\draw[->] (A) edge[in=145,out=-145,loop] (A);
			\draw[->] (B) edge[in=35,out=-35,loop] (B);
			\end{tikzpicture}\]
		\item
			When $n=4$, there is a single additional optimal quiver, which we call the \emph{dumbbell quiver}:
			\[\begin{tikzpicture}
				\node[shape=circle] (A) at (0,0) {$\bullet$};
				\node[shape=circle] (B) at (1.5,0) {$\bullet$};
				\draw[->] (A) edge[in=145,out=-145,loop] (A);
				\draw[->] (B) edge[in=35,out=-35,loop] (B);
				\draw[->,bend left=15] (A) edge (B);
				\draw[->,bend left=15] (B) edge (A);
			\end{tikzpicture}\]
	\end{itemize}
\end{proposition}

\begin{proof}
	\Cref{prop:dim-xqu} gives a formula for $\dim \A^\diag_\Qu$, reducing the proposition to a purely combinatorial statement.
	The proposed quivers do reach the proposed maximum, establishing the lower bound $\max_{\Qu \in \Bal_n} \dim \A^\diag_\Qu \geq \lfloor n^2/3 \rfloor + 1$.
	We prove by induction on $n$ that this is indeed the maximum, and that the quivers reaching that maximum are exactly the proposed ones.
	We leave aside the cases $n=1$ and $n=2$, which are easily checked.
	Let $n > 2$, and assume that for all $n' < n$ and for all $\Qu' \in \mathcal \Bal_{n'}$ we have $\dim \A^\diag_{\Qu'} \leq \lfloor {n'}^2 / 3 \rfloor + 1$.
	We consider a quiver $\Qu \in \Bal_n$ satisfying $\dim \A^\diag_\Qu \geq \lfloor n^2 / 3 \rfloor + 1$.

	We first show that $\Qu$ is connected.
	For this, notice that~$\dim\A^\diag_\Qu$ is additive with respect to unions of vertex-disjoint quivers.
	By the induction hypothesis and since the function $\eta(n) := \lfloor n^2/3 \rfloor+1$ is strictly superadditive on positive integers with the single exception of the equality $\eta(1)+\eta(1)=\eta(2)$, we cannot reach or beat $\lfloor n^2/3 \rfloor+1$ if there are at least two connected components (recall that we have assumed $n>2$).

	Now, let $\ell$ be an integer, and consider a subquiver $C \subseteq \Qu$ which is a union of any number of vertex-disjoint cycles whose lengths sum to $\ell$ (for example, $C$ can be a single $\ell$-cycle), thus consisting of~$\ell$ vertices and $\ell$ edges.
	Removing from the quiver $\Qu$ the edges of~$C$ and the vertices which have become isolated, we obtain a balanced quiver~$\Qu \setminus C$ with $n-\ell$ edges.
	We have, by \Cref{prop:dim-xqu}:
	\begin{align}
		\dim \A^\diag_\Qu - \dim \A^\diag_{\Qu \setminus C}
		= \; &
		\underbrace{
			|\{
				i \in V(C)
				\;\vert\;
				d_{\Qu}(i) = 1
			\}|
		}_{\textnormal{vertices which have become isolated}}
		+
		\sum_{i \in V(C)}
			\Bigl[
				d_{\Qu}(i)^2 - (d_{\Qu}(i) - 1)^2
			\Bigr]
		\nonumber
		\\ &
		-
		\sum_{(i \to j) \in C}
			\Bigl[
				|\Qu(i,j)|^2 -
				(|\Qu(i,j)|-1)^2
			\Bigr]
		\nonumber
		\\
		= \; &
		|\{
			i \in V(C)
			\;\vert\;
			d_{\Qu}(i) = 1
		\}|
		+
		2
		\sum_{i \in V(C)}
			d_{\Qu}(i)
		-
		2
		\sum_{(i \to j) \in C}
			|\Qu(i,j)|.
		\label{eqn:dimension-remove-cycles}
	\end{align}
	By hypothesis, $\dim \A^\diag_\Qu \geq \lfloor n^2/3 \rfloor+1$.
	By the induction hypothesis, $\dim \A^\diag_{\Qu \setminus C} \leq  \lfloor (n-\ell)^2 / 3 \rfloor + 1$.
	Therefore:
	\begin{equation}
		\label{eqn:octoproof-use-indhyp}
		\dim \A^\diag_\Qu
		- \dim \A^\diag_{\Qu \setminus C}
		\geq
		\left\lfloor \frac{n^2}3 \right\rfloor
		-
		\left\lfloor \frac{(n-\ell)^2}3 \right\rfloor
		\geq
		\frac{n^2-2}3
		-
		\frac{n^2-2n\ell + \ell^2}3
		=
		\frac{2n\ell - \ell^2 - 2}3
	\end{equation}
	We clearly have $|\{ i \in V(C) \;\vert\; d_{\Qu}(i) = 1 \}| \leq \ell$.
	If $|\{ i \in V(C) \;\vert\; d_{\Qu}(i) = 1 \}| = \ell$, then $C$ is a union of connected components of~$\Qu$, hence $\Qu = C$ as $\Qu$ is connected (in particular, $C$ is a single cycle in this case), so $\ell = r = n$, but then $\dim \A^\diag_\Qu = n$ is less than $\lfloor n^2/3 \rfloor + 1$ since $n > 2$.
	Therefore, we actually have $|\{ i \in V(C) \;\vert\; d_{\Qu}(i) = 1 \}| \leq \ell - 1$, and so:
	\begin{equation}
		\label{octoproof-bound-isolated}
		|\{
			i \in V(C)
			\;\vert\;
			d_{\Qu}(i) = 1
		\}|
		+
		2
		\sum_{i \in V(C)}
			d_{\Qu}(i)
		-
		2
		\sum_{(i \to j) \in C}
			|\Qu(i,j)|
		\leq
		(\ell - 1)
		+ 2 n
		- 2 \ell
		= 2n - \ell - 1.
	\end{equation}
	Combining \Cref{octoproof-bound-isolated,eqn:octoproof-use-indhyp,eqn:dimension-remove-cycles}, we must then have:
	\[
		2n - \ell - 1
		\geq
		\frac{2n\ell - \ell^2 - 2}3.
	\]
	Multiplying by $3$ and rearranging, this becomes
	\[
		\underbrace{(\ell - 2n)}_{< 0}
		(\ell - 3) \geq 1,
	\]
	which is only possible if $\ell \leq 2$.
	We have thus shown:
	\begin{equation}
		\label{eq:disjoint-union-of-cycles}
		\textnormal{There is no union of vertex-disjoint cycles of~$\Qu$ whose lengths sum to $3$ or more.}
		\tag{C}
	\end{equation}
	Since $\Qu$ is balanced, it can be written as a union of (not necessarily disjoint) cycles.
	By (\ref{eq:disjoint-union-of-cycles}), only $1$-cycles (i.e., loops) and $2$-cycles can occur.
	In particular, for any two vertices $i\neq j$, the number of edges $i\to j$ equals the number of edges $j\to i$.
	We use the notation $i \overset{\alpha}{\longleftrightarrow} j$ as a shortcut for $\alpha$ edges $i \to j$ and~$\alpha$ edges $j \to i$ (this still counts as $2\alpha$ edges!).
	By (\ref{eq:disjoint-union-of-cycles}), there can be at most two vertices with loops.
	We distinguish two cases:
	\begin{description}
		\item[Case 1: There are two vertices $i,j$ with loops.]~\\
			Then, (\ref{eq:disjoint-union-of-cycles}) implies that any $2$-cycle contains both $i$ and $j$.
			As~$\Qu$ is connected, $i$ and $j$ are the only vertices and $\Qu$ looks as follows:
			\[\begin{tikzpicture}
				\node (A) at (0,0) {$i$};
				\node (B) at (1.75,0) {$j$};
				\draw[->] (A) edge[in=145,out=-145,loop] node[left,inner sep=0.4ex] {\scriptsize$\alpha$} (A);
				\draw[->] (B) edge[in=35,out=-35,loop] node[right,inner sep=0.4ex] {\scriptsize$\beta$} (B);
				\draw[<->] (A) edge node[above,inner sep=0.4ex] {\scriptsize$\gamma$} (B);
			\end{tikzpicture}\]
			where~$\alpha + 2 \gamma + \beta = n$ and $\alpha, \beta, \gamma \geq 1$ (this is only possible if $n \geq 4$).
			We have
			\[
				\dim\A^\diag_\Qu =
				2 + (\alpha + \gamma)^2 + (\beta + \gamma)^2 - \alpha^2 - 2 \gamma^2 - \beta^2
				= 2 + 2\gamma(\alpha + \beta)
				= 2 + 2\gamma(n-2\gamma)
				= -4\gamma^2 +2n\gamma + 2.
			\]
			This polynomial in $\gamma$ reaches its real maximum at $\gamma = \frac n4$ with maximal value $\frac {n^2}4 + 2$, which is strictly smaller than $\lfloor n^2/3 \rfloor+1$ as soon as $n \geq 5$, contradicting the hypothesis.
			For $n = 4$, the only possibility is $\alpha=\beta = \gamma=1$, corresponding to the dumbbell quiver.
		\item[Case 2: There is at most one vertex with loops.]~\\
			If there is a vertex $i$ with loops, then all $2$-cycles must contain the vertex $i$ according to (\ref{eq:disjoint-union-of-cycles}).
			Otherwise, (\ref{eq:disjoint-union-of-cycles}) still implies that any two $2$-cycles must share a vertex, and since there cannot be a $3$-cycle, all $2$-cycles then share some common vertex $i$.%
			\footnote{Say $i\leftrightarrow j$ and $i\leftrightarrow k$ are two $2$-cycles sharing a vertex $i$, with $j\neq k$.
			Any $2$-cycle not containing the same vertex~$i$ would need to be $j\leftrightarrow k$.
			But then, $\Qu$ would contain the $3$-cycle $i\to j\to k\to i$, contradicting (\ref{eq:disjoint-union-of-cycles}).}
			Either way, since $\Qu$ is connected, we see that $\Qu$ is of the following form:
			
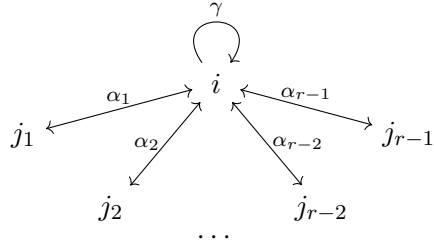
\[\begin{tikzpicture}[scale=1.2]
				\node[shape=circle] (O) at (0,0) {$i$};
				\node (A) at (-15:2.2) {$j_{r-1}$};
				\node (B) at (-50:1.8) {$j_{r-2}$};
				\node (C) at (-130:1.8) {$j_2$};
				\node (D) at (-165:2.2) {$j_1$};
				\node at (-90:1.7) {$\cdots$};
				\draw[<->] (O) edge node[inner sep=0.4ex,above] {\scriptsize$\alpha_{r-1}$} (A);
				\draw[<->] (O) edge node[inner sep=0.4ex,right] {\scriptsize$\alpha_{r-2}$} (B);
				\draw[<->] (O) edge node[inner sep=0.4ex,left] {\scriptsize$\alpha_{2}$} (C);
				\draw[<->] (O) edge node[inner sep=0.4ex,above] {\scriptsize$\alpha_{1}$} (D);
				\draw[->] (O) edge[in=55,out=125,loop] node[inner sep=0.4ex,above] {\scriptsize$\gamma$} (O);
			\end{tikzpicture}\]
			with $\gamma \geq 0$, $\alpha_1 ,\ldots, \alpha_{r-1} \geq 1$, and $\gamma + 2 \sum_i \alpha_i = n$.
			We then have
			\[\textstyle
				\dim \A^\diag_\Qu = r + (n- \sum_i \alpha_i)^2 + \sum_i \alpha_i^2 - (n-2\sum_i \alpha_i)^2 - 2\sum_i \alpha_i^2 = r + 2n \sum_i \alpha_i - 3 (\sum_i\alpha_i)^2 - \sum_i \alpha_i^2.
			\]
			Let $S:=\sum_i \alpha_i$.
			We have $S \geq r-1$ and $\sum_i \alpha_i^2 \geq S$, and thus:
			\[
				\dim \A^\diag_\Qu
				\leq
				(S+1) + 2nS - 3 S^2 - S
				= - 3 S^2 + 2 n S + 1.
			\]
			This upper bound is a quadratic polynomial in $S$ whose real maximum is at $\frac n3$, thus reaching its integer maximum exactly when $S=[\frac n3]$, in which case this evaluates precisely to~$\lfloor n^2/3 \rfloor + 1$.
			As, by hypothesis, we have $\dim \A^\diag_\Qu \geq \lfloor n^2/3 \rfloor + 1$, all inequalities above must be equalities.
			In particular, we have $S = r-1$ so $\alpha_1 = \ldots = \alpha_{r-1} = 1$, and $S = [n/3]$ so $r = S + 1 = [n/3] + 1$.
			The quiver $\Qu$ is then precisely the octopus quiver.
			\qedhere
	\end{description}
\end{proof}

The following subsections are dedicated to describing the irreducible components of~$\A^\diag_\Qu$ for the quivers $\Qu$ maximizing the dimension.

\subsection{The special case $n=2$}
\label{subsn:diag-2}

In the case $n=2$, \Cref{prop:octopus} shows that there are two isomorphism classes of quivers $\Qu \in \Bal_2$ for which $\dim \A^\diag_\Qu$ reaches the maximal value $2$, namely:
\[
	\Oct_2
	=
	\begin{tikzpicture}[baseline={([yshift=-.5ex]current bounding box.center)}]
		\node (A) at (0,0) {$1$};
		\node (B) at (1.5,0) {$2$};
		\draw[->,bend left=15] (A) edge (B);
		\draw[->,bend left=15] (B) edge (A);
	\end{tikzpicture}
	\qquad\qquad\textnormal{and}\qquad\qquad
	\Oct_1 \sqcup \Oct_1
	=
	\begin{tikzpicture}[baseline={([yshift=-.5ex]current bounding box.center)}]
		\node (A) at (0,0) {$1$};
		\node (B) at (1.5,0) {$2$};
		\draw[->] (A) edge[in=145,out=-145,loop] (A);
		\draw[->] (B) edge[in=35,out=-35,loop] (B);
	\end{tikzpicture}
\]

\begin{proposition}
	\label{lem:irreducible-two}
	The two-dimensional set $\A^\diag_{\Oct_2} \sqcup \A^\diag_{\Oct_1 \sqcup \Oct_1}$ has exactly $p^2$ irreducible components, which are all fixed by $\sigma$.
\end{proposition}

\begin{proof}
	We let $X := \A^\diag_{\Oct_2} \sqcup \A^\diag_{\Oct_1 \sqcup \Oct_1}$.
	Note that $Y := \Y_{\Oct_2} = \Y_{\Oct_1 \sqcup \Oct_1}$ is the space of pairs of distinct elements of $\bar\F_p$.
	An element of $Z := \Z_{\Oct_2} \sqcup \Z_{\Oct_1 \sqcup \Oct_1}$ is a pair $(V_1, V_2)$ of distinct one-dimensional subspaces of $\bar\F_p^2$ such that either $\sigma(V_1)=V_2$ and $\sigma(V_2) = V_1$ (for $\Z_{\Oct_2}$), or $\sigma(V_1)=V_1$ and $\sigma(V_2) = V_2$ (for $\Z_{\Oct_1 \sqcup \Oct_1}$); this can be summed up by saying that the unordered pair $\{V_1, V_2\}$ is $\sigma$-invariant.
	We have already counted such unordered pairs in \Cref{thm:count-diag-algebras} (cf.~the bijection of \Cref{lem:param-subalg-geomconj}), so we know that there are $p^{2^2-2} = p^2$ such pairs.
	(In this case, it is easier to distinguish between the two cases, giving $\frac12(p^2-p) + \frac12(p^2+p) = p^2$.)
	Thus, the set $Z$ has size $2p^2$.

	Both quivers have automorphism group $\Aut(\Qu)$ isomorphic to $\ZZ/2\ZZ$, corresponding to the permutation of the vertices $1$ and $2$.
	Thus, the maps of \Cref{eqn:surject-to-xqu} combine into a surjective regular $\sigma$-equivariant map $Y \times Z \twoheadrightarrow X$, whose fibers have size $2$.
	Since $Y$ is irreducible, the space $Y \times Z$ has~$2p^2$ irreducible components, over which $\ZZ/2\ZZ$ acts freely (by swapping coordinates of both pairs), and moreover the $\ZZ/2\ZZ$-orbits are unions of $\sigma$-orbits (they form a single orbit for components coming from $\Z_{\Oct_2}$, and two orbits for components coming from $\Z_{\Oct_1 \sqcup \Oct_1}$).
	This implies that $X$ has $p^2$ irreducible components, all of which are fixed by $\sigma$.
\end{proof}

\subsection{A tool to prove irreducibility}
\label{subsn:tool-irred}

We will repeatedly make use of the following lemma to prove the irreducibility of a variety:

\begin{lemma}
	\label{lem:irreducibility-from-fibers}
	Let $f \colon A \to B$ be a regular map between varieties.
	Assume that $A$ is non-empty and has pure dimension $d$.
	Let $B_1,\dots,B_s$ be locally closed subvarieties of~$B$ with $B = \bigsqcup_{i=1}^s B_i$ and for every $x\in B$, let $F_x$ be a variety such that there is an injective regular map $\varphi_x \colon f^{-1}(x) \hookrightarrow F_x$.
	Assume that $B_1$ is irreducible, that $F_x$ is irreducible for all $x\in B_1$, and that
	\begin{alignat*}{2}
		&
		\forall x \in B_1, \quad &&
		\dim F_x + \dim B_1 \leq d,
		\\
		\forall i \in \{2, \ldots, s\}, \quad &
		\forall x \in B_i, \quad &&
		\dim F_x + \dim B_i < d.
	\end{alignat*}
	Then, $A$ is irreducible.
\end{lemma}

\begin{proof}
	The assumptions imply that $\dim f^{-1}(B_i) < d$ for $i=2,\dots,s$.
	Hence, the ($d$-dimensional) irreducible components of~$A$ are in bijection with those of~$A\setminus\bigcup_{i=2}^s f^{-1}(B_i)$.
	We can thus assume without loss of generality that $s=1$, hence $B = B_1$.

	Consider any irreducible component $C$ of~$A$.
	For generic $x\in \overline{f(C)}$, we have
	\[
		d = \dim C
		\leq \dim (f^{-1}(x)\cap C) + \dim f(C)
		\leq \dim F_x + \dim B_1
		\leq d,
	\]
	so all inequalities have to be equalities: $\dim F_x+\dim B_1=d$, the set $f(C)$ is dense in $B=B_1$ (recall that $B_1$ is irreducible), and the set  $\varphi_x(f^{-1}(x)\cap C)$ (which is constructible by Chevalley's theorem) is dense in $F_x$ (recall that $F_x$ is irreducible), hence contains a non-empty open subset of $F_x$.
	
	This implies that, for any two irreducible components $C$ and $C'$ of $A$ and for generic $x \in B$, the set $\varphi_x(f^{-1}(x)\cap C \cap C') = \varphi_x(f^{-1}(x)\cap C) \cap \varphi_x(f^{-1}(x)\cap C')$ contains a non-empty open subset of~$F_x$.
	We have shown that the fibers of the restricted map $f_{| C \cap C'} \colon C \cap C' \to B$ generically have dimension $\dim F_x$ (in particular, that restricted map is dominant), so $\dim(C\cap C') = \dim F_x + \dim B_1 = d$, which implies $C=C'$.
\end{proof}

\subsection{The general case (the octopus variety)}
\label{subsn:octopus-variety}

Let $n\geq3$, and let $\Qu$ be the octopus quiver $\Oct_n$ (defined in \Cref{prop:octopus}).
In this subsection, we show that $\Z_\Qu$ and $\A^\diag_\Qu$ are irreducible (\Cref{lem:irreducible-octopus}).
For this purpose, we will need the following stratification of the Grassmannian:

\paragraph{Stratification of the Grassmannian.}
Let $0\leq a\leq n$.
We partition $\Gr_a(\bar \F_p^n)$ as follows
\[
	\Gr_a(\bar\F_p^n)
	=
	\bigsqcup_{0\leq b\leq\min(a,n-a)}
		\T_{a,b},
\]
where~$\T_{a,b}$ is the subset of~$\Gr_a(\bar\F_p^n)$ consisting of those $a$-dimensional subspaces $V \subseteq \bar\F_p^n$ such that $\dim(\sigma(V)+\sigma^{-1}(V)) = a+b$, or equivalently $\dim(\sigma(V)\cap\sigma^{-1}(V)) = a-b$.

\begin{lemma}
	\label{lem:irreducibility-of-T}
	For any $0 \leq b \leq \min(a,n-a)$, the strata $\T_{a,b}$ satisfy the following properties:
	\begin{enumalpha}
		\item
			$\T_{a,b}$ is locally closed.
		\item
			\label{item:irreducibility-of-T-dim}
			$\T_{a,b}$ is non-empty with pure dimension $b(n-b)$.
		\item
			\label{item:irreducibility-of-T-irred}
			If $b > 0$, then $\T_{a,b}$ is irreducible.
	\end{enumalpha}
\end{lemma}

\noindent
(If $n\not\equiv2\pmod3$, we will only need the ``trivial'' special case $b=\min(a,n-a)$ of point~\ref{item:irreducibility-of-T-irred}.)

\begin{proof}
	Let $X$ be the set of pairs $(V,V')$ of $a$-dimensional subspaces of~$\bar\F_p^n$ with $\dim(V+V')=a+b$ (equivalently, $\dim(V\cap V')=a-b$).
	The subset $X \subseteq \Gr_a(\bar\F_p^n)\times\Gr_a(\bar\F_p^n)$ is locally closed, cf.\ \Cref{eq:fixed-intersection-space}.

	\begin{enumalpha}
		\item
			Since $\dim(\sigma(V)+\sigma^{-1}(V)) = \dim(\sigma^2(V)+V)$, the set $\T_{a,b}$ is locally closed as the pullback of~$X$ under the regular map $V\mapsto(V,\sigma^2(V))$.
		\item
			We use a similar strategy as in the proof of \Cref{lem:dim-Z}.
			First note that $\T_{a,b}$ is isomorphic to the variety $\tilde\T_{a,b}$ of those pairs $(V,V')\in X$ satisfying $\sigma^2(V) = V'$.

			Let $C_1 := \bar\F_p^b$, $C_2 := \bar\F_p^{a-b}$, $C_3 := \bar\F_p^b$, $C_4 := \bar\F_p^{n-a-b}$, and $C := C_1\oplus C_2\oplus C_3\oplus C_4$.
			We parametrize pairs $(V,V')\in X$ via the regular map
			\[
				f \colon \Isom(C, \bar\F_p^n) \to X,
				\qquad
				E \mapsto
				\bigl(
					E(C_1\oplus C_2),
					E(C_2\oplus C_3)
				\bigr).
			\]
			This map is surjective and its fibers are isomorphic to the variety
			\begin{align*}
				F :={}& \{E\in\GL(C) \mid E(C_1\oplus C_2) = C_1\oplus C_2\textnormal{ and }E(C_2\oplus C_3) = C_2\oplus C_3\} \\
				={}& \{E\in\GL(C) \mid E(C_1)\subseteq C_1\oplus C_2\textnormal{ and }E(C_2)=C_2\textnormal{ and }E(C_3)\subseteq C_2\oplus C_3\}
			\end{align*}
			of dimension
			\begin{align*}
				\dim F
				&= \dim C_1\cdot\dim(C_1\oplus C_2) + (\dim C_2)^2 + \dim C_3\cdot\dim(C_2\oplus C_3) + \dim C_4\cdot\dim C \\
				&= ba + (a-b)^2 + ba + (n-a-b)n \\
				&= a^2+(n-a)n-b(n-b).
			\end{align*}
			(That a generic linear endomorphism $E \colon C \to C$ such that $E(C_1) \subseteq C_1 \oplus C_2$, $E(C_2) \subseteq C_2$ and $E(C_3) \subseteq C_2 \oplus C_3$ satisfies $E(C_2) = C_2$ and is invertible follows from the fact that $F$ is Zariski open in the vector space of such endomorphisms, and is non-empty as it contains the identity.)

			Let $E \in \Isom(C, \bar\F_p^n)$ and let $(V,V') = f(E)$.
			We have $\sigma^2(V) = V'$ if and only if the automorphism $\bwp(E):= E^{-1}\sigma^2(E) \in \GL(C)$ (with $\sigma^2(E)$ defined analogously to $\sigma(E)$ in the proof of \Cref{lem:dim-Z}) lies in the irreducible variety
			\[
				S := \{A\in\GL(C) \mid A(C_1\oplus C_2) = C_2\oplus C_3\}
			\]
			of dimension
			\begin{align*}
				\dim S
				&= \dim(C_1\oplus C_2)\cdot\dim(C_2\oplus C_3) + \dim(C_3\oplus C_4)\cdot\dim C
				= a^2 + (n-a)n.
			\end{align*}
			(As above, generic invertibility comes from the fact that $S$ is non-empty, as it contains the invertible map $C_1\oplus C_2\oplus C_3\oplus C_4 \to C_1\oplus C_2\oplus C_3\oplus C_4, (x,y,z,w) \mapsto (z,y,x,w)$.)

			By \Cref{lem:wp-etale}, $\bwp^{-1}(S)$ is non-empty of pure dimension $\dim S = a^2 + (n-a)n$.
			In particular, $\T_{a,b} \simeq \tilde \T_{a,b} = f(\bwp^{-1}(S))$ is non-empty of pure dimension
			\[
				\dim \bwp^{-1}(S) - \dim F
				=
				a^2 + (n-a)n - a^2 - (n-a)n + b(n-b)
				=
				b(n-b).
			\]
		\item
			We use downward induction on $a$.
			(The case $a=n$ is vacuous.)
			The case $b=\min(a,n-a)$ is clear since by~\ref{item:irreducibility-of-T-dim}, $\T_{a,b}$ is then a subvariety of dimension $a(n-a)$ of the irreducible variety~$\Gr_a(\bar\F_p^n)$ of dimension $a(n-a)$, hence it is dense, hence itself irreducible.
			We can therefore assume that $0<b<\min(a,n-a)$.

			We are going to apply \Cref{lem:irreducibility-from-fibers} to the regular map
			\[
				f
				\colon
				\T_{a,b}
				\to
				\bigsqcup_{0\leq c\leq\min(a+b, \, n-a-b)}
					\T_{a+b,c}
				=
				\Gr_{a+b}(\bar\F_p^n)
			\]
			sending $V$ to $W:=\sigma^2(V)+V$.
			For any $W\in \T_{a+b,c}$, the fiber $f^{-1}(W)$ is contained in the set of $a$-dimensional subspaces $V$ of the $(a+b-c)$-dimensional vector space $W\cap\sigma^{-2}(W)$.
			In particular, the fiber is empty unless $a\leq a+b-c$, so $c\leq b$.

			Let $0\leq c\leq\min(b, \, n-a-b)$ and $W \in \T_{a+b,c}$.
			The fiber $f^{-1}(W)$ embeds into the irreducible variety $\Gr_a(W\cap\sigma^{-2}(W))\simeq\Gr_a(\bar\F_p^{a+b-c})$ and we have
			\begin{align*}
			&\dim \T_{a,b} - \dim \Gr_a(\bar\F_p^{a+b-c}) - \dim \T_{a+b,c} \\
			\stackrel{\mathclap{\textnormal{\ref{item:irreducibility-of-T-dim}}}}{=}{}& b(n-b) - a(b-c) - c(n-c) \\
			={}& (b-c)(n-a-b-c).
			\end{align*}
			The right-hand side is positive for all $c$ except $c = \min(b,n-a-b)$, for which it is zero.
			For this value of~$c$, the assumption $0<b<\min(a,n-a)$ implies that $a+b>a$ and $c>0$, so~$\T_{a+b,c}$ is irreducible by the induction hypothesis.

			The claim follows by applying \Cref{lem:irreducibility-from-fibers} to the regular map $f$.
			\qedhere
	\end{enumalpha}
\end{proof}

\begin{remark}
	For $b=0$, the variety $\T_{a,0}$ consists of those $a$-dimensional subspaces $V \subseteq \bar \F_p^n$ such that $\sigma(V)=\sigma^{-1}(V)$, or, equivalently, of the finitely many $a$-dimensional subspaces of~$\bar\F_p^n$ defined over~$\F_{p^2}$.
	In particular, $\T_{a,0}$ is not irreducible unless $a\in\{0,n\}$.
\end{remark}

\begin{proposition}
	\label{lem:irreducible-octopus}
	Let $n\geq3$ and $k=[n/3]$.
	Let $\Qu = \Oct_n$ be the octopus quiver with $k+1$~vertices and~$n$~edges.
	Then, the sets $\Z_{\Qu}$ and $\A^\diag_\Qu$ are irreducible.
\end{proposition}

\[\begin{tikzpicture}
	\node (O) at (0,0) {$0$};
	\node (A) at (-15:2.2) {$k$};
	\node (B) at (-50:1.8) {$k-1$};
	\node (C) at (-130:1.8) {$2$};
	\node (D) at (-165:2.2) {$1$};
	\node at (-90:1.7) {$\cdots$};
	\draw[->,bend left=7] (O) edge (A);
	\draw[<-,bend right=7] (O) edge (A);
	\draw[->,bend left=10] (O) edge (B);
	\draw[<-,bend right=10] (O) edge (B);
	\draw[->,bend left=10] (O) edge (C);
	\draw[<-,bend right=10] (O) edge (C);
	\draw[->,bend left=7] (O) edge (D);
	\draw[<-,bend right=7] (O) edge (D);
	\draw[->] (O) edge[in=55,out=125,loop] node[above,inner sep=0.4ex] {\scriptsize$n-2k$} (O);
\end{tikzpicture}\]

\begin{proof}
By \iref{lem:dim-Z}{item:dim-Z-normal}, we have
\[
	\dim \Z_\Qu
	= \sum_i d_{\Qu}(i)^2 - \sum_{i,j} |\Qu(i,j)|^2
	= (n-k)^2 + k - (n-2k)^2 - 2k
	= k(2n-3k-1).
\]

Points in $\Z_{\Qu}$ correspond to tuples $(V_0,V_1,\dots,V_k)$ of subspaces of~$\bar\F_p^n$ of respective dimensions $n-k,1,\dots,1$ together spanning $\bar\F_p^n$ such that $\dim(V_0\cap\sigma(V_0)) = n-2k$ and $V_1,\dots,V_k \subseteq \sigma(V_0)\cap\sigma^{-1}(V_0)$ and $V_i\neq\sigma(V_j)$ for all $i,j\in\{1,\dots,k\}$.
We are going to apply \Cref{lem:irreducibility-from-fibers} to the regular map
\[
	f \colon
	\Z_{\Qu} \to
	\bigsqcup_{0\leq b\leq\min(n-k,k)}
		\T_{n-k,b}
	= \Gr_{n-k}(\bar\F_p^n)
\]
sending $(V_0,V_1,\dots,V_k)$ to $V_0$.

Let $0\leq b\leq\min(n-k,k)$ and consider an arbitrary $V_0\in \T_{n-k,b}$.
The fiber $f^{-1}(V_0)$ consists of tuples $(V_1,\dots,V_k)$ of linearly independent one-dimensional subspaces of the $(n-k-b)$-dimensional vector space $\sigma(V_0)\cap\sigma^{-1}(V_0)$.
In particular, the fiber is empty unless $k\leq n-k-b$, i.e., $b\leq n-2k$.
We now assume that $b\leq n-2k$.
The fiber $f^{-1}(V_0)$ embeds into the irreducible variety $\Bigl(\P\bigl(\sigma(V_0)\cap\sigma^{-1}(V_0)\bigr)\Bigr)^k \simeq \Bigl(\P(\bar\F_p^{n-k-b})\Bigr)^k$, and by \iref{lem:irreducibility-of-T}{item:irreducibility-of-T-dim} we have
\begin{align*}
& \dim \Z_\Qu - \dim \Bigl(\P(\bar\F_p^{n-k-b})\Bigr)^k - \dim \T_{n-k,b} \\
={}& k(2n-3k-1) - k(n-k-b-1) - b(n-b) \\
={}& (n-2k-b)(k-b).
\end{align*}
The right-hand side is positive for all $b$ except $b = \min(n-2k,k)$, for which it is zero.
For this value of~$b$, the assumption $n\geq3$ together with the definition $k=[n/3]$ imply that $b>0$, so $\T_{n-k,b}$ is irreducible by \iref{lem:irreducibility-of-T}{item:irreducibility-of-T-irred}.
By \Cref{lem:irreducibility-from-fibers}, the variety $\Z_\Qu$ is irreducible.
Since $\Y_{\Qu}$ and $\Z_{\Qu}$ are irreducible, so is their product and therefore so is the image $\A^\diag_\Qu$.
\end{proof}

\subsection{The special case $n=4$ (the dumbbell variety)}
\label{subsn:diag-dumbbell}

When $n=4$, \Cref{prop:octopus} shows that there are two isomorphism classes of quivers $\Qu \in \Bal_4$ such that $\dim \A^\diag_\Qu$ reaches the maximal value $6$, namely the octopus quiver $\Oct_4$ (for which $\A^\diag_{\Oct_4}$ is irreducible by \Cref{lem:irreducible-octopus}), and the dumbbell quiver $\Qu$:
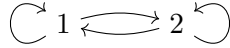
\[\begin{tikzpicture}
	\node (A) at (0,0) {$1$};
	\node (B) at (1.5,0) {$2$};
	\draw[->] (A) edge[in=145,out=-145,loop] (A);
	\draw[->] (B) edge[in=35,out=-35,loop] (B);
	\draw[->,bend left=15] (A) edge (B);
	\draw[->,bend left=15] (B) edge (A);
\end{tikzpicture}\]
The goal of this subsection is to prove:
\begin{proposition}
	\label{lem:irreducible-dumbbell}
	When $\Qu$ is the dumbbell quiver, the sets $\Z_\Qu$ and $\A^\diag_\Qu$ are irreducible.
\end{proposition}

By \iref{lem:dim-Z}{item:dim-Z-normal}, the set $\Z_\Qu$ has pure dimension $4$.
The points of~$\Z_\Qu$ correspond to pairs $V = (V_1,V_2)$ of two-dimensional subspaces of~$\bar\F_p^4$ such that $V_1 \oplus V_2 = \bar\F_p^4$ and $\dim(V_i\cap\sigma(V_j)) = 1$ for each $i,j \in \{1,2\}$.
For any $V = (V_1, V_2) \in \Z_\Qu$, define the one-dimensional vector spaces
\[
	L_1(V) := V_1\cap\sigma(V_1)
	\qquad\textnormal{and}\qquad
	L_2(V) := V_2\cap\sigma(V_2),
\]
the three-dimensional vector space
\[
	W(V) := V_1 + \sigma(V_2),
\]
and the vector spaces
\[
	U(V) := W(V) \cap \sigma(W(V)),
\]
and
\[
	M(V) := U(V)\cap\sigma(U(V)) = W(V)\cap\sigma(W(V))\cap\sigma^2(W(V)).
\]

Since $W(V)$ has codimension $1$ in $\bar\F_p^4$, we have $\dim U(V) \geq 2$ and $\dim M(V) \geq 1$.
The space~$W(V)$ is not defined over~$\F_p$ as otherwise we would have $V_1+V_2\subseteq W(V)\subsetneq\bar\F_p^4$.
This implies that $U(V) \subsetneq W(V)$ is two-dimensional.

Note that
\begin{equation}
	\label{eq:Li_in_U}
	L_1(V) \subseteq U(V)
	\qquad\textnormal{and}\qquad
	L_2(V) \subseteq \sigma^{-1}(U(V)).
\end{equation}

\paragraph{Strategy.}

Our strategy of proof for \Cref{lem:irreducible-dumbbell} is as follows: we show that for a generic element~$V$ of any irreducible component of~$\Z_\Qu$, none of the subspaces $L_1(V)$, $L_2(V)$, $U(V)$, $M(V)$ are defined over~$\F_p$.
Disregarding those ``exceptional'' $V$ for which any of these subspaces are defined over~$\F_p$, we show that $M(V)$ is one-dimensional, and that the fibers of the map $V \mapsto M(V)$ embed into one-dimensional subvarieties of~$\P^1(\bar\F_p)\times\P^1(\bar\F_p)$.
Using Newton polygons, we show that these one-dimensional varieties are generically irreducible.
Finally, we conclude using \Cref{lem:irreducibility-from-fibers}.

\begin{lemma}
	\label{lem:dumbbell-lambda}
	Consider the regular map
	\[
		\lambda \colon \Z_\Qu \to \P(\bar\F_p^4)\times\P(\bar\F_p^4),\qquad
		V \mapsto \bigl(L_1(V),L_2(V)\bigr).
	\]
	Let $F_\lambda$ be the closed subset of~$\P(\bar\F_p^4) \times \P(\bar\F_p^4)$ corresponding to pairs $(L_1, L_2)$ such that at least one of~$L_1$ or~$L_2$ is defined over~$\F_p$, and let $\Z_\Qu' := \Z_\Qu \setminus \lambda^{-1}(F_\lambda)$.
	Then:
	\begin{enumalpha}
	\item
		\label{item:dumbbell-lambda-1}
		The closed subset $\lambda^{-1}(F_\lambda)$ of~$\Z_\Qu$ is at most three-dimensional.
	\item
		\label{item:dumbbell-lambda-2}
		For any $V\in \Z_\Qu'$, we have:
		\begin{enumroman}
		\item\label{item:dumbbell-lambda-2-Vi}
			$
				V_i = L_i(V) \oplus \sigma^{-1}(L_i(V))
			$
			for each $i\in\{1,2\}$.
		\item\label{item:dumbbell-lambda-2-U}
			$
				U(V) + \sigma^{-1}(U(V)) + \sigma^{-2}(U(V))
				= \bar\F_p^4
			$.
		\item\label{item:dumbbell-lambda-2-M}
			The vector space $M(V)$ is one-dimensional.
		\end{enumroman}
	\end{enumalpha}
\end{lemma}
\begin{proof}
	~
	\begin{enumalpha}
		\item
			As both cases are symmetric, we can focus on the preimage of the space of pairs where~$L_1$ is defined over~$\F_p$.
			By \iref{lem:dim-Z}{item:dim-Z-special}, this preimage has dimension strictly less than $\dim\Z_\Qu = 4$.
		\item
			\begin{enumroman}
			\item
				By definition, $V_i \supseteq L_i(V)+\sigma^{-1}(L_i(V))$.
				By hypothesis, $L_i(V)$ is a one-dimensional space not defined over~$\F_p$, so it has trivial intersection with $\sigma^{-1}(L_i(V))$, so the right-hand side is a direct sum and has dimension $2 = \dim V_i$, so the inclusion is an equality.
			\item
				Combining \ref{item:dumbbell-lambda-2-Vi} with \Cref{eq:Li_in_U}, we obtain $V_1+V_2 \subseteq U(V)+\sigma^{-1}(U(V))+\sigma^{-2}(U(V))$.
				The left-hand side is $\bar\F_p^4$ since $V\in\Z_\Qu$.
			\item
				From \ref{item:dumbbell-lambda-2-Vi}, we see that the two-dimensional vector space $U(V)$ is not defined over $\F_p$.
				Thus, the vector space $M(V) = U(V)\cap\sigma(U(V))$ is at most one-dimensional, but it is also at least one-dimensional since it equals $W(V)\cap\sigma(W(V))\cap\sigma^2(W(V))$ and $\dim W(V) = 3$.
				\qedhere
			\end{enumroman}
	\end{enumalpha}
\end{proof}

Consider the regular map
\[
	\upsilon \colon \Z_\Qu' \to \Gr_2(\bar\F_p^4),\qquad
	V \mapsto U(V).
\]
If $U = \langle v,u\rangle$ is any two-dimensional subspace of $\bar\F_p^4$, then \Cref{eq:Li_in_U} shows that there is a regular map
\[
	\varphi_{v,u} \colon \upsilon^{-1}(U) \to \P^1(\bar\F_p)\times\P^1(\bar\F_p),\qquad
	V \mapsto \bigl([r_1:s_1],[r_2:s_2]\bigr)
\]
uniquely characterized by
\begin{align}\label{eq:Li_from_rs}
	L_1(V) &= \langle r_1 v + s_1 u \rangle
	&\textnormal{and}&&
	L_2(V) &= \langle r_2 \sigma^{-1}(v) + s_2 \sigma^{-1}(u) \rangle.
\end{align}
This map $\varphi_{v,u}$ is injective as $V_i = L_i(V) \oplus \sigma^{-1}(L_i(V))$ by \iiref{lem:dumbbell-lambda}{item:dumbbell-lambda-2}{item:dumbbell-lambda-2-Vi}.

Let $S$ be the (dense open) subset of~$\bar\F_p^4$ consisting of those $m\in\bar\F_p^4$ for which the vectors $\sigma^i(m)$ for $i=0,\dots,3$ are linearly independent, and let $g \colon S \to \bar\F_p^4$ be the map sending $m$ to the unique tuple $(c_0,\dots,c_3) \in \bar\F_p^4$ satisfying $\sigma^4(m) = \sum_{i=0}^3 c_i \sigma^i(m)$.
The map $g$ is regular by Cramer's rule.
Finally, for any $\underline c = (c_0,\dots,c_3)\in \bar\F_p^4$, define the following (one-dimensional) closed subset $D_{\underline c} \subseteq \P^1(\bar\F_p) \times \P^1(\bar\F_p)$:
\begin{equation}
	\label{def:dc}
	D_{\underline c} :=
	\Bigl\{
		\bigl(
			[r_1:s_1],[r_2:s_2]
		\bigr)
		\;\Big\vert\;
		- c_0 r_1^{p+1}r_2^{p+1} + c_1 r_1^{p+1}r_2^ps_2 - c_2 r_1^{p+1}s_2^{p+1} + c_3 r_1^ps_1s_2^{p+1} + s_1^{p+1}s_2^{p+1} = 0
	\Bigr\}
\end{equation}

\begin{lemma}
	\label{lem:dumbbell-mu}
	Consider the regular map (see \iiref{lem:dumbbell-lambda}{item:dumbbell-lambda-2}{item:dumbbell-lambda-2-M})
	\[
		\mu \colon \Z_\Qu' \to \P(\bar\F_p^4),\qquad
		V \mapsto M(V).
	\]
	Let $F_\mu$ be the closed (finite) subset of~$\P(\bar\F_p^4)$ corresponding to subspaces $M$ which are defined over~$\F_p$, and let $\Z_\Qu'' := \Z_\Qu' \setminus \mu^{-1}(F_\mu)$.
	\begin{enumalpha}
		\item\label{item:dumbbell-mu-1}
			If $M\in F_\mu$, then the closed subset $\mu^{-1}(M)$ of~$\Z_\Qu'$ is at most three-dimensional.
		\item
			\label{item:dumbbell-mu-2}
			If $M=\langle m\rangle\in\P(\bar\F_p^4) \setminus F_\mu$, then the closed subset $\mu^{-1}(M)$ of~$\Z_\Qu'$ is at most one-dimensional.
			More specifically, if $\mu^{-1}(M)$ is non-empty, then $m$ lies in $S$ and there is an injective regular map $\mu^{-1}(M) \hookrightarrow D_{g(m)}$ (where~$D_{g(m)}$ is as in \Cref{def:dc}).
	\end{enumalpha}
\end{lemma}

\begin{proof}
	The proofs of \ref{item:dumbbell-mu-1} and \ref{item:dumbbell-mu-2} are very similar, the main difference being that for fixed $M$, in~\ref{item:dumbbell-mu-2}, there is only one possible vector space $U(V)$, whereas in~\ref{item:dumbbell-mu-1}, there is a two-dimensional set of possible vector spaces $U(V)$.
	\begin{enumalpha}
	\item
		Since $M$ is defined over~$\F_p$, we pick a $\sigma$-invariant generator $m \in (M \cap \F_p^4) \setminus \{0\}$ of~$M$.
		For any $V \in \mu^{-1}(M)$, the two-dimensional vector space $U(V)$ contains $M$ by definition.
		As $\{U \in \Gr_2(\bar\F_p^4) \mid M\subseteq U\}\simeq\P(\bar\F_p^4/M)$ is two-dimensional, it suffices to show that the image of the injective map $\varphi_{m,u} \colon \upsilon^{-1}(U) \to \P^1(\bar\F_p)\times\P^1(\bar\F_p)$ is at most one-dimensional for any $U=\langle m,u\rangle$ containing $M$.
		Let $U = \langle m,u\rangle$ be a two-dimensional subspace of~$\bar\F_p^4$ containing $M$, and assume that $\upsilon^{-1}(U)$ is non-empty.
		
		By \iiref{lem:dumbbell-lambda}{item:dumbbell-lambda-2}{item:dumbbell-lambda-2-U}, this implies $\bar\F_p^4 = U + \sigma^{-1}(U) + \sigma^{-2}(U) = \sigma^{-2}(\langle m,u,\sigma(u),\sigma^2(u)\rangle)$, so the vectors $m,u,\sigma(u),\sigma^2(u)$ form a basis of $\bar\F_p^4$.
		Write $\sigma^3(u) = \sum_{i=0}^2 c_i \sigma^i(u) + c_3 m$ with $c_0,\dots,c_3\in\bar\F_p$.

		For any $V \in \upsilon^{-1}(U)$, letting $\varphi_{m,u}(V) = ([r_1:s_1],[r_2:s_2])$, since $\sigma(V_1)\cap V_2\neq0$, we must have $\sigma^3(V_1)\cap\sigma^2(V_2)\neq0$, where according to \iiref{lem:dumbbell-lambda}{item:dumbbell-lambda-2}{item:dumbbell-lambda-2-Vi} and \Cref{eq:Li_from_rs}:
		\begin{align*}
			\sigma^3(V_1)
			&= \sigma^3(L_1(V)) + \sigma^2(L_1(V))
			= \langle
			r_1^{p^3} m + s_1^{p^3} \sigma^3(u),\quad
			r_1^{p^2} m + s_1^{p^2} \sigma^2(u)
			\rangle \\
			&= \langle
			(r_1^{p^3} + c_3 s_1^{p^3}) m + c_0 s_1^{p^3} u + c_1 s_1^{p^3} \sigma(u) + c_2 s_1^{p^3} \sigma^2(u),\quad
			r_1^{p^2} m + s_1^{p^2} \sigma^2(u)
			\rangle, \\
			\sigma^2(V_2)
			&= \sigma^2(L_2(V)) + \sigma(L_2(V))
			= \langle
			r_2^{p^2} m + s_2^{p^2} \sigma(u),\quad
			r_2^p m + s_2^p u
			\rangle.
		\end{align*}
		Writing everything in terms of the basis $(m,u,\sigma(u),\sigma^2(u))$, this means that the matrix
		\[
			\begin{pmatrix}
				r_1^{p^3} + c_3 s_1^{p^3} & c_0 s_1^{p^3} & c_1 s_1^{p^3} & c_2 s_1^{p^3} \\
				r_1^{p^2} & & & s_1^{p^2} \\
				r_2^{p^2} & & s_2^{p^2} \\
				r_2^p & s_2^p
			\end{pmatrix}
		\]
		must be singular, so its determinant must vanish.
		This determinant is a non-zero polynomial in $r_1,s_1,r_2,s_2$ (it always involves the summand~$r_1^{p^3}s_1^{p^2}s_2^{p^2+p}$), which shows that the image of~$\varphi_{m,u}$ in $\P^1(\bar\F_p)\times\P^1(\bar\F_p)$ is indeed at most one-dimensional.
	\item
		Let $M=\langle m\rangle\in\P(\bar\F_p^4) \setminus F_\mu$.
		If $\mu^{-1}(M)$ is empty, the claims are clear, so we assume that $\mu^{-1}(M)$ is non-empty.
		For any $V\in\mu^{-1}(M)$, we have $U(V) \supseteq M(V) + \sigma^{-1}(M(V))$ by definition; since the one-dimensional space~$M(V) = M$ is not defined over~$\F_p$ and since $U(V)$ is two-dimensional, we in fact have
		\begin{equation*}
			U(V) = M(V) \oplus \sigma^{-1}(M(V)) = \langle m , \sigma^{-1}(m) \rangle,
		\end{equation*}
		so $\mu^{-1}(M) = \upsilon^{-1}(U)$ where $U := \langle m, \sigma^{-1}(m) \rangle$.
		
		By \iiref{lem:dumbbell-lambda}{item:dumbbell-lambda-2}{item:dumbbell-lambda-2-U}, and because $\upsilon^{-1}(U)$ is non-empty, we have $\bar\F_p^4 = U + \sigma^{-1}(U) + \sigma^{-2}(U) = \sigma^{-3}\bigl(\langle m,\sigma(m),\sigma^2(m),\sigma^3(m)\rangle\bigr)$, so the vectors $m,\sigma(m),\sigma^2(m),\sigma^3(m)$ form a basis of $\bar\F_p^4$, i.e., $m$ lies in $S$.
		Let $(c_0,\dots,c_3) := g(m) \in \bar\F_p^4$, so that by definition $\sigma^4(m) = \sum_{i=0}^3 c_i \sigma^i(m)$.

		For any $V \in \mu^{-1}(M)$, letting $\varphi_{m,\sigma^{-1}(m)}(V) = ([r_1:s_1],[r_2:s_2])$, since $\sigma(V_1)\cap V_2\neq0$, we must have $\sigma^4(V_1)\cap\sigma^3(V_2)\neq0$, where according to \iiref{lem:dumbbell-lambda}{item:dumbbell-lambda-2}{item:dumbbell-lambda-2-Vi} and \Cref{eq:Li_from_rs}:
		\begin{align*}
			\sigma^4(V_1)
			&= \sigma^4(L_1(V)) + \sigma^3(L_1(V))
			= \langle
			r_1^{p^4} \sigma^4(m) + s_1^{p^4} \sigma^3(m),\quad
			r_1^{p^3} \sigma^3(m) + s_1^{p^3} \sigma^2(m)
			\rangle \\
			&= \langle
			c_0 r_1^{p^4} m + c_1 r_1^{p^4} \sigma(m) + c_2 r_1^{p^4} \sigma^2(m) + (c_3 r_1^{p^4} + s_1^{p^4}) \sigma^3(m),\quad
			r_1^{p^3} \sigma^3(m) + s_1^{p^3} \sigma^2(m)
			\rangle, \\
			\sigma^3(V_2)
			&= \sigma^3(L_2(V)) + \sigma^2(L_2(V))
			= \langle
			r_2^{p^3} \sigma^2(m) + s_2^{p^3} \sigma(m),\quad
			r_2^{p^2} \sigma(m) + s_2^{p^2} m
			\rangle.
		\end{align*}
		Writing everything in terms of the basis $(m,\sigma(m),\sigma^2(m),\sigma^3(m))$, this means that the matrix
		\[
			\begin{pmatrix}
				c_0 r_1^{p^4} & c_1 r_1^{p^4} & c_2 r_1^{p^4} & c_3 r_1^{p^4} + s_1^{p^4} \\
				& & s_1^{p^3} & r_1^{p^3} \\
				& s_2^{p^3} & r_2^{p^3} \\
				s_2^{p^2} & r_2^{p^2}
			\end{pmatrix}
		\]
		must be singular, so its determinant
		\[
			-c_0r_1^{p^4+p^3}r_2^{p^3+p^2} + c_1r_1^{p^4+p^3}r_2^{p^3}s_2^{p^2} - c_2r_1^{p^4+p^3}s_2^{p^3+p^2} + c_3r_1^{p^4}s_1^{p^3}s_2^{p^3+p^2} + s_1^{p^4+p^3}s_2^{p^3+p^2}
		\]
		must vanish.
		Letting $\tau \colon \P^1(\bar\F_p)\times\P^1(\bar\F_p) \to \P^1(\bar\F_p)\times\P^1(\bar\F_p)$ be the bijective regular map $([r_1:s_1],[r_2:s_2]) \mapsto ([r_1^{p^3}:s_1^{p^3}],[r_2^{p^2}:s_2^{p^2}])$, we have shown that the image of the injective regular map $\tau\circ\varphi_{m,\sigma^{-1}(m)} \colon \mu^{-1}(M) \to \P^1(\bar\F_p^4)\times\P^1(\bar\F_p^4)$ is contained in $D_{(c_0,\dots,c_3)} = D_{g(m)}$.
		\qedhere
	\end{enumalpha}
\end{proof}

\begin{lemma}
	\label{lem:irreducible-polynomial}
	There is a non-empty open subset $O' \subseteq \bar\F_p^4$ such that, for all $\underline c\in O'$, the closed subset~$D_{\underline c} \subseteq \P^1(\bar\F_p) \times \P^1(\bar\F_p)$ is irreducible.
\end{lemma}

\begin{proof}
	Let $f$ be the following bihomogeneous polynomial in the variables $r_1,s_1,r_2,s_2$, with coefficients in~$\F_p(c_0,\dots,c_3)$:
	\[
		f = - c_0 r_1^{p+1}r_2^{p+1} + c_1 r_1^{p+1}r_2^ps_2 - c_2 r_1^{p+1}s_2^{p+1} + c_3 r_1^ps_1s_2^{p+1} + s_1^{p+1}s_2^{p+1}
	\]
	Let $L = \bar{\F_p(c_0,\dots,c_3)}$.
	By \cite[\href{https://stacks.math.columbia.edu/tag/0559}{Lemma 0559}]{stacks-project}, it suffices to prove that the subscheme of~$\P^1_L\times\P^1_L$ defined by $f=0$ is irreducible, i.e., that $f$ is irreducible as a bihomogeneous polynomial over $L$.
	We will show this by specializing to $c_0=0$.
	Assume by contradiction that there are non-constant bihomogeneous polynomials $g,h \in L[r_1,s_1,r_2,s_2]$ such that $f=gh$.
	Let $v$ be an extension of the $c_0$-adic valuation on~$\F_p(c_0,\dots,c_3)$ to $L$ and let $\mathfrak p\subset\mathcal O\subset L$ be the corresponding maximal ideal and valuation ring.
	We have $\mathcal O/\mathfrak p = \bar{\F_p(c_1,\dots,c_3)}$.
	Since the coefficients of~$f$ lie in~$\mathcal O$, we can by Gauss' lemma assume without loss of generality that the coefficients of~$g$ and~$h$ also lie in~$\mathcal O$.%
	\footnote{
		Let $a$ and $b$ be the smallest valuations of coefficients of $g$ and $h$, respectively.
		Considering the lexicographically minimal monomials whose coefficients have these valuations and expanding the product, one can see that some coefficient of $gh$ has valuation $a+b$.
		Since all coefficients of $f=gh$ lie in $\mathcal O$, this means that $a+b\geq0$.
		Dividing $g$ by an element of valuation $a$ and multiplying $h$ by the same element, we can ensure that the coefficients of $g$ and $h$ lie in $\mathcal O$.
	}

	The \emph{Newton polygon} $\NP(a) \subset \mathbb R^2$ of a bihomogeneous polynomial $a = \sum_{i,j} k_{ij} r_1^i s_1^{d-i} r_2^j s_2^{e-j}$ with coefficients in an integral domain is the convex hull of the points $(i,j)\in\ZZ_{\geq0}^2$ with $k_{ij}\neq0$.
	For any two such polynomials $a,b$, the Newton polygon $\NP(ab)$ is the Minkowski sum of~$\NP(a)$ and $\NP(b)$.

	Over~$\mathcal O/\mathfrak p = \bar{\F_p(c_1,\dots,c_3)}$, we have
	\[
		(f \bmod \mathfrak p) = c_1 r_1^{p+1}r_2^ps_2 - c_2 r_1^{p+1}s_2^{p+1} + c_3 r_1^ps_1s_2^{p+1} + s_1^{p+1}s_2^{p+1}.
	\]
	The Newton polygon of~$f$ is the (solid) triangle with corners $(0,0)$, $(p+1,0)$, $(p+1,p+1)$ and the Newton polygon of~$(f\bmod\mathfrak p)$ is the (dashed) triangle with corners $(0,0)$, $(p+1,0)$, $(p+1,p)$.

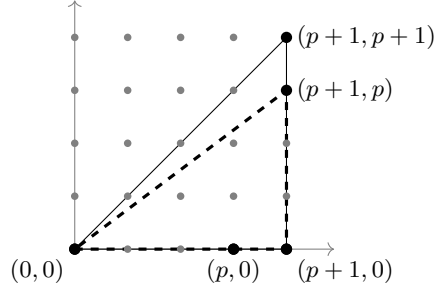
	\begin{center}
	\begin{tikzpicture}[x=7mm,y=7mm]
	\draw[gray,->] (0,0) -- (4.9,0);
	\draw[gray,->] (0,0) -- (0,4.7);
	\draw (0,0) -- (4,0) -- (4,4) -- cycle;
	\draw[dashed,very thick] (0,0) -- (4,0) -- (4,3) -- cycle;
	\foreach \x in {0,...,4} {
		\foreach \y in {0,...,4} {
		\fill[gray] (\x,\y) circle[radius=0.3ex];
		}
	}
	\fill (0,0) circle[radius=0.45ex];
	\fill (3,0) circle[radius=0.45ex];
	\fill (4,0) circle[radius=0.45ex];
	\fill (4,3) circle[radius=0.45ex] node[right] {\footnotesize $(p+1,p)$};
	\fill (4,4) circle[radius=0.45ex] node[right] {\footnotesize $(p+1,p+1)$};
	\fill (0,0) circle[radius=0.45ex] node[below left] {\footnotesize $(0,0)$};
	\fill (4,0) circle[radius=0.45ex] node[below right] {\footnotesize $(p+1,0)$};
	\fill (3,0) circle[radius=0.45ex] node[below] {\footnotesize $(p,0)$};
	\end{tikzpicture}
	\end{center}

	The line segment $[(0,0),(p+1,p)]$ contains no integer lattice points other than its endpoints.
	Since $\NP(f\bmod\mathfrak p) = \NP(g\bmod\mathfrak p) + \NP(h\bmod\mathfrak p)$ and the corners of the Newton polygons $\NP(g\bmod\mathfrak p)$ and $\NP(h\bmod\mathfrak p)$ are non-negative integer lattice points, it follows that the Newton polygon of one of the factors (say $\NP(g\bmod\mathfrak p)$) contains a translate of that line segment.
	Moreover, as all other edges of~$\NP(f \bmod \mathfrak p)$ are either horizontal or vertical, so are the other edges of~$\NP(g\bmod\mathfrak p)$.
	The only possibility is that $\NP(g \bmod \mathfrak p) = \NP(f\bmod \mathfrak p)$, and then $\NP(g)\supseteq\NP(g\bmod\mathfrak p) = \NP(f\bmod\mathfrak p)$.

	We have $\NP(f) = \NP(g) + \NP(h)$, but the triangle $\NP(f)$ does not contain any proper translate of~$\NP(f\bmod\mathfrak p) \subseteq \NP(g)$, so $\NP(h)=\{(0,0)\}$, i.e., $h$ is a monomial of the form $k s_1^d s_2^e$.
	Clearly, such a monomial can only divide $f$ if $d=e=0$, so $h$ must be constant.
\end{proof}

\begin{corollary}
	\label{cor:dumbbell-mu-irred}
	There is a dense open subset $O$ of~$\P(\bar\F_p^4)$ such that for all $M\in O$, there is an injective regular map from the fiber $\mu^{-1}(M)$ to a one-dimensional irreducible variety.
\end{corollary}
\begin{proof}
	All fibers of the map $g$ are finite since they are cut out by the non-trivial polynomial equations $m_i^{p^4} = \sum_{i=0}^3 c_i m_i^{p^i}$ in the coordinates $m_1,\dots,m_4$ of $m$.
	Since $\dim S = 4 = \dim\bar\F_p^4$ and $\bar\F_p^4$ is irreducible, this implies that $g$ is dominant.
	We have seen in \iref{lem:dumbbell-mu}{item:dumbbell-mu-2} that for any $m\in S$ (in particular, $\langle m\rangle$ is not defined over $\F_p$), there is an injective regular map $\mu^{-1}(\langle m\rangle) \to D_{g(m)}$.
	(This is obviously true if $\mu^{-1}(\langle m\rangle)$ is empty.)
	Now, let $O'$ be as in \Cref{lem:irreducible-polynomial}, so that $D_{g(m)}$ is irreducible when $g(m) \in O'$.
	The claim follows, taking~$O$ to be any dense open subset of the image of $g^{-1}(O') \subseteq S \subseteq \bar\F_p^4$ under the regular map $\bar\F_p^4 \to \P(\bar\F_p^4)$, $m \mapsto \langle m \rangle$.
	(The preimage $g^{-1}(O')$ is non-empty and open since $O'$ is non-empty and open and $g$ is dominant.
	Hence, its (constructible) image in $\P(\bar\F_p^4)$ is dense, so it contains a dense open subset.)
\end{proof}

\begin{proof}[Proof of \Cref{lem:irreducible-dumbbell}]
	The set $\Z_\Qu$ has pure dimension $4$ by \iref{lem:dim-Z}{item:dim-Z-normal}.
	Thus, \iref{lem:dumbbell-lambda}{item:dumbbell-lambda-1} and \iref{lem:dumbbell-mu}{item:dumbbell-mu-1} imply that the inclusions $\Z''_\Qu \subseteq \Z'_\Qu \subseteq \Z_\Qu$ are dense, so it suffices to prove that~$\Z''_\Qu$ is irreducible.
	For this, fix $O$ as in \Cref{cor:dumbbell-mu-irred} (which is three-dimensional and whose complement is at most two-dimensional) and apply \Cref{lem:irreducibility-from-fibers} to the map $\mu \colon \Z''_\Qu \to \P(\bar\F_p^4) \setminus F_\mu$.
	(The fiber~$\mu^{-1}(M)$ embeds in a one-dimensional variety by \iref{lem:dumbbell-mu}{item:dumbbell-mu-2}, and that variety can be taken to be irreducible when $x \in O$ by \Cref{cor:dumbbell-mu-irred}.)
\end{proof}

\subsection{Conclusion}
\label{subsn:diag-conclusion}

\begin{theorem}
	\label{thm:1-diag}
	For any finite field $\F_q\supseteq\F_p$, we have
	\[
		|\A^\diag (\F_q)|
		=
		c^\diag(p,n)
		\cdot
		q^{\lfloor n^2/3\rfloor + 1}
		+
		O_{p,n}\Bigl(
			q^{\lfloor n^2/3\rfloor + 1/2}
		\Bigr),
		\textnormal{ where }
		c^\diag(p,n)
		=
		\begin{cases}
			p^2 & \textnormal{if }n=2, \\
			2 & \textnormal{if }n=4, \\
			1 & \textnormal{if }n\notin\{2,4\}.
		\end{cases}
	\]
\end{theorem}

\begin{proof}
	We have seen above that $\A^\diag$ is a disjoint union of the finitely many constructible $\sigma$-invariant
	subsets~$\A^\diag_\Qu$.
	For all quivers $\Qu$ with $\dim \A^\diag_\Qu\leq\lfloor n^2/3\rfloor$, we have $|\A^\diag_\Qu (\F_q)| = O_{p,n}(q^{\lfloor n^2/3\rfloor})$ by the Schwartz--Zippel bound \cite[Lemma~1]{lang-weil}.
	\Cref{prop:octopus} classifies the remaining quivers and shows that they all satisfy $\dim \A^\diag_\Qu=\lfloor n^2/3\rfloor+1$.
	In \Cref{lem:irreducible-two,lem:irreducible-octopus,lem:irreducible-dumbbell}, we have computed the number of irreducible components of~$\A^\diag_\Qu$ in these cases, shown that they are all fixed by $\sigma$, and that the total number of irreducible components of dimension $\lfloor n^2/3\rfloor+1$ is precisely $c^\diag(p,n)$.
	The claim then follows from the Lang--Weil bound \cite[Theorem~1]{lang-weil}.
\end{proof}

\section{Towards general matrices commuting with their Frobenius}
\label{sn:general}

In this section, we relate the size of $\A(\F_q)$ to the numbers $d(M)$ defined in \Cref{eq:def-dM}, i.e., we prove \Cref{prop:counting-by-shape} (which implies \Cref{thm:1-all}).
To this end, we associate to any matrix in~$\Mat_n(\bar\F_p)$ a \emph{Jordan shape}, encoding the sizes of all Jordan blocks associated to the eigenvalues.

\paragraph{Jordan shapes.}
A \emph{Jordan shape} of size $n$ is a pair $\Sh = (V,e)$ consisting of a finite set $V$ and a map $e \colon V \times \N\to\mathbb Z_{\geq 0}$ such that $e(i,1)\geq1$ and $e(i,1)\geq e(i,2)\geq\cdots$ for all $i\in V$ and such that $\sum_{i\in V}\sum_{k\geq1} e(i,k) = n$.
An \emph{isomorphism} between Jordan shapes $\Sh = (V,e)$ and $\Sh' = (V',e')$ is a bijection $\pi \colon V \to V'$ such that $e(i,k) = e'(\pi(i),k)$ for all $i\in V$ and $k\geq1$.
We let $\JS_n$ be the (finite) set of isomorphism classes of Jordan shapes of size $n$.

\begin{definition}
	To any matrix $M\in\Mat_n(\bar\F_p)$, we associate a Jordan shape $\Sh_M=(V_M,e_M)$ of size $n$ as follows: the set $V_M$ consists of the eigenvalues of~$M$; for each eigenvalue $\lambda$ and each $k\geq1$, we let
	\[
		e_M(\lambda,k) := \dim\leftl(\ker(M-\lambda I_n)^k / \ker(M-\lambda I_n)^{k-1}\rightr),
	\]
	be the number of Jordan blocks of size at least $k$ for this eigenvalue.
\end{definition}
Two matrices $M$ and $M'$ are conjugate if and only if they have equal Jordan shapes, i.e., $V_M = V_{M'}$ and $e_M = e_{M'}$.
Two matrices having \emph{isomorphic} Jordan blocks, by contrast, may not have the same eigenvalues---for instance, $M$ and $\sigma(M)$ always have isomorphic Jordan shapes via $\pi \colon \lambda \mapsto \sigma(\lambda)$.

\paragraph{The space of matrices with a given Jordan shape commuting with their Frobenius.}

For any Jordan shape $\Sh\in\JS_n$, we define the subset $\A_\Sh \subseteq \A$ of matrices $M\in\Mat_n(\bar\F_p)$ such that~$M$ commutes with $\sigma(M)$ and such that~$\Sh_M$ is isomorphic to~$\Sh$.
(The vertices of the Jordan shape~$\Sh = (V,e)$ are purely combinatorial: the $\card V$ eigenvalues of a matrix~$M \in \A_\Sh$ are subject to no constraints other than distinctness.)
Clearly,
\[
	\A = \bigsqcup_{\Sh\in\JS_n} \A_\Sh.
\]

\begin{remark}
	The sets $\A_\Sh$ for $\Sh\in\JS_n$ defined here are related to the constructible sets $\A^\diag_\Qu$ for $\Qu\in\Bal_n$ defined in \Cref{subsn:matrices-and-quivers} as follows: if the shape $\Sh=(V,e)$ is one that corresponds to diagonalizable matrices (i.e., $e(i,2)=0$ for all $i\in V$), then $\A_\Sh$ is the union of the sets $\A^\diag_\Qu$ over all quivers $\Qu\in\Bal_n$ whose vertex set $V(\Qu)$ is $V$ and whose degrees satisfy $d_\Qu(i)=e(i,1)$ for all $i\in V$.
\end{remark}

For any matrix $M\in\Mat_n(\bar\F_p)$, denote by $\Cent M$ its centralizer and by $\Cl M$ its conjugacy class.
Note that $\Cent M$ is a subalgebra of~$\Mat_n(\bar\F_p)$ and that $\Cl M$ is a constructible subset of~$\Mat_n(\bar\F_p)$.

Now, fix a shape $\Sh=(V,e)$, say with $V = \{1,\dots,r\}$.
Let $\Y_\Sh \subseteq \bar\F_p^r$ be the (open) subset formed of tuples $\lambda=(\lambda_1,\dots,\lambda_r)$ of distinct elements of~$\bar\F_p$.
For any $\lambda\in \Y_\Sh$, we define a matrix $A_{\Sh,\lambda}$ of shape~$\Sh$ as follows: $A_{\Sh,\lambda}$ is the matrix in Jordan normal form having $e(i,k)-e(i,k+1)$ Jordan blocks of size $k$ associated to each eigenvalue~$\lambda_i$, where we put the Jordan blocks for eigenvalue~$\lambda_i$ before those for eigenvalue $\lambda_j$ if $i<j$, and we order blocks with the same eigenvalue by their size.

\begin{lemma}
	\label{lem:cent-indep-lambda}
	For any $\lambda,\lambda'\in \Y_\Sh$, we have $\Cent A_{\Sh,\lambda} = \Cent A_{\Sh,\lambda'}$.
	We denote the corresponding subalgebra of~$\Mat_n(\bar\F_p)$ by~$\Cent \Sh$.
\end{lemma}

\begin{proof}
	For any $i\in\{1,\dots,n\}$, the generalized eigenspace of~$A_{\Sh,\lambda}$ with eigenvalue $\lambda_i$ is also the generalized eigenspace of~$A_{\Sh,\lambda'}$ with eigenvalue $\lambda_i'$.
	Denote this common generalized eigenspace by~$G_i$.
	We have $A_{\Sh,\lambda'}v = A_{\Sh,\lambda}v + (\lambda'_i-\lambda_i)v$ for all $v\in G_i$.
	The claim follows since any matrix commuting with $A_{\Sh,\lambda}$ or $A_{\Sh,\lambda'}$ preserves the generalized eigenspaces.
\end{proof}

\begin{remark}
	The centralizer $\Cent \Sh$ admits an explicit description (some coefficients have to vanish, and some others must be equal), see \cite[Chap.~VIII, \S 2]{gantmakher}.
	Its dimension is
	$
		\sum_{i=1}^r
			\sum_{k \geq 1}
				e(i,k)^2
	$.
\end{remark}

\begin{corollary}
	The set of matrices $U\in\GL_n(\bar\F_p)$ such that $A_{\Sh,\lambda}$ commutes with $UA_{\Sh,\tilde\lambda}U^{-1}$ does not depend on the choice of~$\lambda,\tilde\lambda\in \Y_\Sh$.
	We denote this closed subset of~$\GL_n(\bar\F_p)$ by $\D_\Sh$.
\end{corollary}

\begin{proof}
	This follows from \Cref{lem:cent-indep-lambda} due to the following equivalences:
	\[
		\begin{gathered}[b]
			\begin{tikzcd}
				A_{\Sh,\lambda}\textnormal{ commutes with }UA_{\Sh,\tilde\lambda}U^{-1}
				\dar[Leftrightarrow] \rar[Leftrightarrow]
				&
				UA_{\Sh,\tilde\lambda}U^{-1} \in \Cent \Sh \text{ (independent of $\lambda$)}
				\\
				U^{-1}A_{\Sh,\lambda}U\textnormal{ commutes with }A_{\Sh,\tilde\lambda}
				\rar[Leftrightarrow]
				&
				U^{-1}A_{\Sh,\lambda}U \in \Cent \Sh \text{ (independent of $\tilde\lambda$)}
			\end{tikzcd}\\[-2.5\dp\strutbox]
		\end{gathered}\qedhere
	\]
\end{proof}

\begin{proposition}
	\label{prop:dimension-of-XS}
	For any Jordan shape $\Sh=(V,e)$, the set $\A_\Sh$ is a constructible subset of~$\Mat_n(\bar\F_p)$ of dimension $|V| + \dim \D_\Sh - \dim \Cent \Sh$.
\end{proposition}

\begin{proof}
	As before, we may assume that $V = \{1,\dots,r\}$.
	Let $M$ be a matrix such that we have an isomorphism $\pi \colon \Sh \to \Sh_M$.
	Then, taking $\lambda := (\pi(1),\dots,\pi(r))$, we see that $M$ must be conjugate to~$A_{\Sh,\lambda}$.
	Write $M = UA_{\Sh,\lambda} U^{-1}$.
	Then, $M$ commutes with $\sigma(M) = \sigma(U) A_{\Sh,\sigma(\lambda)} \sigma(U)^{-1}$ if and only if~$A_{\Sh,\lambda}$ commutes with $(U^{-1}\sigma(U))A_{\Sh,\sigma(\lambda)}(U^{-1}\sigma(U))^{-1}$, i.e., if and only if $\bwp(U) := U^{-1}\sigma(U)$ lies in~$\D_\Sh$.
	We have shown that the regular map
	\[
		\Y_\Sh \times \bwp^{-1}(\D_\Sh) \to \Mat_n(\bar\F_p),\qquad (\lambda,U) \mapsto UA_{\Sh,\lambda}U^{-1}
	\]
	has image $\A_\Sh$.
	In particular, $\A_\Sh$ is constructible by Chevalley's theorem.
	Each fiber is the union of~$|\Aut(\Sh)|$ sets of the form $\{(\lambda, US) \mid S \in (\Cent \Sh)^\times\}$ where $(\lambda,U)\in\Y_\Sh\times\bwp^{-1}(\D_\Sh)$, hence has dimension $\dim (\Cent \Sh)^\times = \dim \Cent \Sh$.
	By \Cref{lem:wp-etale}, we have $\dim \bwp^{-1}(\D_\Sh) = \dim \D_\Sh$.
	Thus,
	\[
		\dim \A_\Sh = \dim \Y_\Sh + \dim \bwp^{-1}(\D_\Sh) - \dim \Cent \Sh = |V| + \dim \D_\Sh - \dim \Cent \Sh.
		\qedhere
	\]
\end{proof}

\begin{lemma}
	\label{lem:dim-cent-cl}
	For any matrix $M\in\Mat_n(\bar\F_p)$ with Jordan shape $\Sh_M\simeq\Sh$, the subset $\Cent M\cap\Cl M$ of~$\Mat_n(\bar\F_p)$ has pure dimension $\dim \D_\Sh - \dim \Cent\Sh$.
\end{lemma}

\begin{proof}
	Replacing $M$ by a conjugate, we can assume without loss of generality that $M = A_{\Sh,\lambda}$ for some $\lambda\in \Y_\Sh$.
	Then, the regular map
	\[
		\D_\Sh \to \Mat_n(\bar\F_p),\qquad U\mapsto UA_{\Sh,\lambda}U^{-1}
	\]
	has image $\Cent M\cap\Cl M$, and each fiber is a left coset of~$(\Cent\Sh)^\times$.
\end{proof}

For any matrix $M\in\Mat_n(\bar\F_p)$, let
\[
	d(M)
	:=
	(\textnormal{number of distinct eigenvalues of~$M$})
	+
	\dim\bigl(
		{\Cent M\cap\Cl M}
	\bigr).
\]

\begin{proposition}
	\label{prop:counting-by-shape}
	Let $\Sh$ be any Jordan shape of size $n$ and let $M\in\Mat_n(\bar\F_p)$ be any matrix with $\Sh_M\simeq\Sh$.
	Then, there is an integer $c\geq1$ and a finite field $\F_{q_0} \supseteq \F_p$ such that:
	\begin{enumalpha}
		\item
			$|\A_\Sh(\F_q)| \leq c\cdot q^{d(M)} + O_{p,n}(q^{d(M)-1/2})$ for all finite fields $\F_q \supseteq \F_p$.
		\item
			$|\A_\Sh(\F_q)| = c\cdot q^{d(M)} + O_{p,n}(q^{d(M)-1/2})$ for all finite fields $\F_q \supseteq \F_{q_0}$.
	\end{enumalpha}
\end{proposition}

\begin{proof}
	By \Cref{prop:dimension-of-XS} and \Cref{lem:dim-cent-cl}, the constructible set $\A_\Sh$ has dimension $d(M)$.
	The claims follow from the Lang--Weil bound \cite[Theorem~1]{lang-weil}, where~$c$ is the number of $d(M)$-dimensional irreducible components of~$\A_\Sh$, and $\F_{q_0}$ is any finite field over which these irreducible components are all defined.
\end{proof}

\Cref{thm:1-all} follows from \Cref{prop:counting-by-shape} by summing over all shapes corresponding to non-dia\-go\-na\-li\-za\-ble matrices.

\begin{remark}
	We do not know whether for any $n\geq3$, there is a non-diagonalizable matrix $M$ for which~$d(M)$ is larger than or equal to the exponent $\lfloor n^2/3\rfloor+1$ obtained for diagonalizable matrices in \Cref{thm:1-diag}.
	The largest value which we have been able to obtain for nilpotent matrices is $d(M) = \lfloor n(n-1)/3\rfloor + 1$, for the nilpotent matrix $M$ with one Jordan block of size $[n/3]+1$ and $n-[n/3]-1$ Jordan blocks of size $1$.
\end{remark}

\begin{remark}
	Some computations of $\dim\bigl({\Cent \Sh \cap \Cl A_{\Sh, \lambda}}\bigr)$ exist in the literature, centered mostly around the nilpotent case (i.e., $r=1$, $\lambda = (0)$).
	In particular, in that case, an upper bound is given by the dimension of the space of nilpotent matrices in~$\Cent \Sh$, that is
	$
		\sum_{k \geq 1} e(0, k)^2 - e(0,1)
	$,
	and equality holds if and only if $\Sh$ is \emph{self-large}, meaning that $e(0,k) - e(0, k+2) \leq 1$ for all~$k$, i.e., any two distinct Jordan blocks have sizes differing by at least $2$.
	(In that case, a generic nilpotent matrix in $\Cent \Sh$ automatically has shape $\Sh$.)
	We refer to \cite{panyushev} for details concerning this case.
\end{remark}

\section{Matrices with eigenspaces defined over~$\F_p$ and commuting with their Frobenius}
\label{subsn:eigenspaces-def-over-fp}

In this section, in order to illustrate the principle described in \Cref{sn:general}, we deal with a special case: the set $\A^\eig$ of matrices $M \in \A$ whose eigenspaces $\ker(M-\lambda I_n)$ are all defined over~$\F_p$.
Specifically, we determine the asymptotics of $|\A^\eig(\F_q)|$, i.e., we prove \Cref{thm:eig} (which is \Cref{thm:eig-intro}).

This case is made accessible by the following observation:

\begin{lemma}
	\label{lem:commute-incl-kernels}
	Let $A$ and $B$ be two commuting matrices in $\Mat_n(\bar\F_p)$.
	\begin{enumroman}
		\item
			\label{lem:commute-incl-kernels-i}
			If $\ker A \subseteq \ker B$, then $\ker A^k \subseteq \ker B^k$ for all $k \geq 1$.
		\item
			\label{lem:commute-incl-kernels-ii}
			If $\ker (A-\lambda I_n) = \ker (B-\lambda I_n)$ for all $\lambda \in \bar \F_p$, then $\ker (A-\lambda I_n)^k = \ker (B-\lambda I_n)^k$ for all $\lambda \in \bar \F_p$ and $k \geq 1$.
			In particular, the matrices $A$ and $B$ are conjugate.
	\end{enumroman}
\end{lemma}

\begin{proof}
	We prove~\ref{lem:commute-incl-kernels-i} by induction on $k$: the case $k=1$ is clear.
	Let $k\geq 2$ and assume that $\ker A^{k-1} \subseteq \ker B^{k-1}$.
	Let $x \in \ker A^k$.
	Then, $A(x) \in \ker A^{k-1} \subseteq \ker B^{k-1}$, so $AB^{k-1}(x) = B^{k-1} A(x) = 0$, so $B^{k-1}(x) \in \ker A \subseteq \ker B$, so $B^k(x)=0$.

	For~\ref{lem:commute-incl-kernels-ii}, we reason for a fixed $\lambda$.
	Subtracting $\lambda I_n$ from $A$ and $B$, we may assume that $\lambda = 0$.
	The inclusion $\ker A^k \subseteq \ker B^k$ and the reverse inclusion then both follow from~\ref{lem:commute-incl-kernels-i}.
\end{proof}

\begin{corollary}
	\label{cor:gen-eig-def}
	If $M \in \A^\eig$, then the generalized eigenspaces $\ker(M-\lambda_i I_n)^k$ of~$M$ are all defined over~$\F_p$.
\end{corollary}

\begin{proof}
	The space $\ker(M-\lambda_i I_n)^k$ is defined over $\F_p$ if and only if $\ker(M-\lambda_i I_n)^k = \ker(\sigma(M)-\sigma(\lambda_i) I_n)^k$.
	Since  $M \in \A^\eig$, the matrices $M-\lambda_i I_n$ and $\sigma(M)-\sigma(\lambda_i) I_n$ commute and have equal kernels (this is the case $k=1$).
	Both inclusions between $\ker(M-\lambda_i I_n)^k$ and $\ker(\sigma(M)-\sigma(\lambda_i) I_n)^k$ then follow from \iref{lem:commute-incl-kernels}{lem:commute-incl-kernels-i}.
\end{proof}

For each Jordan shape $\Sh = (\{1,\ldots,r\},e)$, let $\A^\eig_\Sh$ be the subset of~$\A^\eig$ consisting of those matrices whose Jordan shape is isomorphic to $\Sh$.
Note that $\A^\eig = \bigsqcup_{\Sh\in\JS_n} \A^\eig_\Sh$.

\begin{proposition}
	Let $\Sh = (\{1,\ldots,r\},e)$ be a Jordan shape, and let $\lambda = (\lambda_1, \ldots, \lambda_r) \in \Y_\Sh$.
	The set~$\A^\eig_\Sh$ is a non-empty constructible subset of~$\Mat_n(\bar\F_p)$ of pure dimension $r + \dim \E_{\Sh,\lambda}$, where~$\E_{\Sh,\lambda}$ is the following locally closed subset of~$\Mat_n(\bar\F_p)$:
	\[
		\E_{\Sh,\lambda} :=
		\Bigl\{
			B \in \Cent \Sh
			\;\Big\vert\;
			\ker (B - \lambda_i I_n) = \ker (A_{\Sh, \lambda} - \lambda_i I_n)
			\textnormal{ for all }
			1 \leq i \leq r
		\Bigr\}.
	\]
\end{proposition}

\begin{proof}
	The eigenspace $E_i := \ker (A_{\Sh, \lambda} - \lambda_i I_n)$ is by definition defined over $\F_p$.
	If $M = U A_{\Sh, \lambda} U^{-1}$, then the eigenspace $\ker (M - \lambda_i I_n) = U(E_i)$ is defined over~$\F_p$ if and only if $(U^{-1} \sigma(U)) (E_i) = E_i$.
	Letting $\D'_\Sh$ be the set of matrices $U \in \D_\Sh$ such that $U(E_i) = E_i$ for all $i \in \{1, \ldots, r\}$, the same proof as \Cref{prop:dimension-of-XS} shows that~$\A^\eig_\Sh$ has dimension $r + \dim \D'_\Sh - \dim \Cent \Sh$.
	Thus, it suffices to prove that $\E_{\Sh,\lambda}$ has pure dimension $\dim \D'_\Sh - \dim \Cent\Sh$.
	Note that $\E_{\Sh,\lambda} \subseteq \Cl A_{\Sh, \lambda}$ by \iref{lem:commute-incl-kernels}{lem:commute-incl-kernels-ii}.
	The computation is then analogous to the proof of \Cref{lem:dim-cent-cl}.
\end{proof}

We now compute the dimension of~$\E_{\Sh,\lambda}$:

\begin{proposition}
	\label{prop:compute-dim-esh}
	Consider a shape $\Sh = (\{1,\ldots,r\},e)$ and a tuple $\lambda = (\lambda_1, \ldots, \lambda_r) \in \Y_\Sh$.
	Then:
	\begin{enumroman}
		\item
			\label{prop:compute-dim-esh-i}
			We have an isomorphism of varieties
			$\E_{\Sh, \lambda} \simeq \prod_i \E_{\Sh_i, \lambda_i}$, where
			$
				\Sh_i
				:=
				\Bigl(
					\{i\}, \, (i,k) \mapsto e(i,k)
				\Bigr)
			$ is the subshape for the eigenvalue~$\lambda_i$.
		\item
			$\dim \E_{\Sh, \lambda} = \sum_{i=1}^r \sum_{k \geq 1} e(i,k)\cdot e(i,k+1)$
	\end{enumroman}
\end{proposition}

\begin{proof}
	~
	\begin{enumerate}[label=(\roman*)]
		\item
			Let $B \in \E_{\Sh, \lambda}$.
			Since $B$ commutes with $A_{\Sh, \lambda}$, it preserves the generalized eigenspace $G_{\lambda_i}$ for each eigenvalue~$\lambda_i$, inducing maps $B_i \colon G_{\lambda_i} \to G_{\lambda_i}$ which are easily checked to belong to $\E_{\mathcal S_i, \lambda_i}$.
			We have $\bigoplus_{i=1}^r G_{\lambda_i} = \bar\F_p^n$, so $B$ can be reconstructed from the restricted maps $B_i \colon G_{\lambda_i} \to G_{\lambda_i}$.
			We have described two inverse regular maps.
		\item
			By~\ref{prop:compute-dim-esh-i}, we reduce to the case $r=1$.
			Without loss of generality (subtracting $\lambda I_n$ from everything), we have $\lambda = 0$.
			Then, the claim amounts to \Cref{lem:dim-cent-sameker} below with $A = A_{\Sh, 0}$.
			\qedhere
	\end{enumerate}
\end{proof}

\begin{lemma}
	\label{lem:dim-cent-sameker}
	Let $A$ be a nilpotent endomorphism of an $n$-dimensional vector space $V$.
	Let $e_A(k) := \dim \ker A^k - \dim \ker A^{k-1}$ and let
	$\E_A := \{ B \in \Cent(A) \mid \ker B = \ker A \}$.
	Then:
	\[
		\dim \E_A =
		\sum_{k \geq 1}
			e_A(k)\cdot e_A(k+1).
	\]
\end{lemma}

\begin{proof}
	We actually show that the linear subspace $\E_A' := \{ B \in \Cent A \mid \ker B \supseteq \ker A \}$ has the announced dimension.
	Since $\E_A$ is an open subset of~$\E'_A$ (it is defined by the non-vanishing of certain determinants) and is non-empty (it contains $A$), it is Zariski dense and the result shall follow.

	We reason by induction on the dimension $n$ of~$V$.
	Since $A$ is nilpotent, $\im A$ has strictly smaller dimension than $V$, and $\bar A := A|_{\im A}$ is a nilpotent endomorphism of~$\im A$.
	Moreover,
	\begin{align*}
		e_{\bar A}(k)
		&= \dim (\ker A^k \cap \im A) - \dim (\ker A^{k-1} \cap \im A)
		= \dim A(\ker A^{k+1}) - \dim A(\ker A^k) \\
		&= (\dim \ker A^{k+1} - \dim \ker A) - (\dim \ker A^k - \dim \ker A) = e_A(k+1),
	\end{align*}
	so $\dim \E_{\bar A}' = \sum_{k\geq2} e_A(k)\cdot e_A(k+1)$ by the induction hypothesis.
	It therefore suffices to show that the linear map $f \colon \E_A'\to \E'_{\bar A}$ sending $B$ to its restriction $B|_{\im A}$ is surjective and that its kernel has dimension~$e_A(1)\cdot e_A(2)$.

	Consider an endomorphism $\bar B \colon \im A \to \im A$ in $\E_{\bar A}'$.
	The fiber $f^{-1}(\bar B)$ consists of those endomorphisms $B \colon V \to V$ whose restriction to $\im A$ is $\bar B$, which vanish on $\ker A$, and such that the following diagram commutes:
	\[\begin{tikzcd}
		\im A \dar[swap]{\bar B} \rar[<<-]{A} & V \dar[dashed]{B} \\
		\im A \rar[<<-,swap]{A} & V
	\end{tikzcd}\]
	We pick a complement $C$ of~$\im A + \ker A$ in $V$.
	Since $\bar B\in \E_{\bar A}'$,
	restriction to $C$ defines a bijection between $f^{-1}(\bar B)$ and the set of linear maps $B' \colon C \to V$ such that the following diagram commutes:
	\[\begin{tikzcd}
		\im A \dar[swap]{\bar B} \rar[<-]{A} & C \dar[dashed]{B'} \\
		\im A \rar[<<-,swap]{A} & V
	\end{tikzcd}\]
	In particular, the fibers are non-empty (the map $\bar B \circ A$ factors through the surjection $A \colon V\twoheadrightarrow\im A$), so $f$ is surjective.
	Taking $\bar B = 0$, we see that the kernel of~$f$ is isomorphic to the vector space of linear maps $B' \colon C \to\ker A$, of dimension $\dim \ker A\cdot\dim C$.
	The claim follows since $\dim\ker A = e_A(1)$ and
	\begin{align*}
		\dim C
		&= \dim V - \dim (\im A + \ker A)
		= \dim\im A + \dim\ker A - \dim(\im A+\ker A) \\
		&= \dim(\im A \cap \ker A)
		= \dim A(\ker A^2)
		= \dim \ker A^2 - \dim \ker A
		= e_A(2).
		\qedhere
	\end{align*}
\end{proof}

\begin{proposition}
	\label{prop:shapes-maximal-sameker}
	The maximal value of $\dim \A^\eig_\Sh = r + \sum_{i=1}^r \sum_{k \geq 1} e(i,k)\cdot e(i,k+1)$ over shapes~$\Sh$ of size $n$ is $\lfloor n^2/4 \rfloor + 1$, and it is reached exactly for the following shapes (up to isomorphism), where we represent a shape $\Sh = \bigl(\{1,\dots,r\}, \, e\bigr)$ by the tuple $\bigl( (e(1,1), e(1,2), \ldots), \, \ldots, \, (e(r,1), \ldots) \bigr)$, omitting the trailing zeros:
	\begin{center}\setlength{\tabcolsep}{5ex}\begin{tabular}{cl}
		$n$ & \normalfont optimal shapes
		\\ \hline
		$2$
		&
		$\bigl( (1,1) \bigr)$,
		\quad
		$\bigl( (1), (1) \bigr)$
		\\
		$3$
		&
		$\bigl( (2,1) \bigr)$,
		\quad
		$\bigl( (1,1,1) \bigr)$,
		\quad
		$\bigl( (1,1), (1) \bigr)$,
		\quad
		$\bigl( (1), (1), (1) \bigr)$
		\\
		$2m$, $m \geq 2$
		&
		$\bigl( (m,m) \bigr)$
		\\
		$2m + 1$, $m \geq 2$
		&
		$\bigl( (m+1,m) \bigr)$,
		\quad
		$\bigl( (m,m,1) \bigr)$
	\end{tabular}\end{center}
\end{proposition}

\begin{proof}
	First, we consider only shapes with $r=1$.
	Let $\Sh = (\{1\}, e)$, and let $s$ be such that $e(1,s) \neq 0$ and $e(1,s+1) = 0$.
	We have
	\[
		\dim \A^\eig_\Sh
		=
		1+
		\sum_{k=1}^{s-1} e(1,k) e(1,k+1)
		\leq
		1+ \sum_{k=1}^{s-1} e(1,1) e(1,k+1)
		=
		1+ e(1,1) \cdot \bigl(n-e(1,1)\bigr),
	\]
	with equality if and only if $e(1,1) = e(1,2) = \ldots = e(1,s-1)$.
	Since $e(1,1)$ is an integer, $1+ e(1,1) \cdot \bigl(n-e(1,1)\bigr)$ has maximal value $1 + \lfloor n/2 \rfloor \cdot \lceil n/2 \rceil = 1 + \lfloor n^2 / 4 \rfloor$, reached exactly when $e(1,1) \in \{ \lfloor n/2 \rfloor, \lceil n/2 \rceil \}$.
	If $n$ is even, only the shape $\bigl( (\frac n2, \frac n2) \bigr)$ gives equality.
	If $n$ is odd, distinguishing between the two possible values of~$e(1,1)$ gives the two equality cases with $r=1$.

	Now, consider the case of a general shape $\Sh = (\{1, \ldots, r\}, e)$.
	By the case $r=1$, we have
	\[
		r + \sum_{i=1}^r \sum_{k \geq 1} e(i,k)\cdot e(i,k+1)
		=
		\sum_{i=1}^r \left( 1 + \sum_{k \geq 1} e(i,k)\cdot e(i,k+1) \right)
		\leq
		\sum_{i=1}^r \left( 1 + \left\lfloor \frac{\bigl(\sum_{k\geq1} e(i,k)\bigr)^2}4 \right\rfloor \right).
	\]
	However, the function $\eta(n) := \lfloor n^2/4 \rfloor + 1$ is strictly superadditive except for the equalities $\eta(1)+\eta(1) = \eta(2)$ and $\eta(1)+\eta(2) = \eta(1)+\eta(1)+\eta(1) = \eta(3)$.
	Therefore, we must have $r = 1$ if $n > 3$, and the cases $n \in \{2,3\}$ are quickly dealt with.
\end{proof}

It remains only to obtain estimates for $|\A^\eig_\Sh (\F_q)|$ when $\Sh$ is one of the optimal shapes of \Cref{prop:shapes-maximal-sameker}.
For this, we are going to need the following two lemmas:

\begin{lemma}
	\label{lem:count-matrices-vw}
	Let $a\geq1$ and let $\vec v,\vec w\in\F_q^a$ be non-zero vectors.
	The number of matrices $N\in\GL_a(\F_q)$ satisfying $N\vec w = \sigma(N)\vec v$ is $q^{a(a-1)}+O_{p,a}(q^{a(a-1)-1})$ if $\vec v$ and $\sigma(\vec w)$ are linearly independent, and $O_{p,a}(q^{a(a-1)})$ otherwise.
\end{lemma}

\begin{proof}
	Assume first that $\vec v$ and $\sigma(\vec w)$ are linearly independent.
	Replacing $(\vec v,\vec w)$ by $(\sigma(U)\vec v,U\vec w)$ for an appropriate $U\in\GL_a(\F_q)$, we can assume without loss of generality that $\vec v=\vec e_1$ and $\vec w=\vec e_2$ are the first two standard basis vectors.
	Then, $N\vec w=\sigma(N)\vec v$ means that the second column of~$N$ is deduced from the first column by applying~$\sigma$.
	The number of invertible matrices satisfying this condition is as claimed.
	
	Now, assume that $\sigma(\vec w) = \lambda \vec v$ for some $\lambda \in \F_q^\times$.
	Replacing $(\vec v,\vec w)$ by $(\sigma(U)\vec v,U\vec w)$ for an appropriate matrix $U\in\GL_a(\F_q)$, we can assume that $\vec v=\vec e_1$ and $\vec w=\sigma^{-1}(\lambda)\vec e_1$.
	The condition $N\vec w=\sigma(N)\vec v$ then leaves at most $p^a=O_{p,n}(1)$ options for the first column of~$N$.
\end{proof}

\begin{lemma}
	\label{lem:shapes-w}
	Let $m \geq 1$.
	For any filtration of linear subspaces $0=V_0\subseteq \cdots \subseteq V_s=\bar\F_p^m$, where each $V_k$ is defined over~$\F_p$, the number of (nilpotent) matrices $M\in\Mat_m(\F_q)$ commuting with $\sigma(M)$ and such that $\ker M^k = V_k$ for all $k\in\{1,\dots,s\}$ only depends on $q$ and on the numbers $e(k) := \dim (V_k/V_{k-1})$.
	We denote this count by $w_q(e(1),\dots,e(s))$ (we omit trailing zeros in the notation, i.e., this means $e(k) = 0$ for $k \geq s+1$).
	Moreover:
	\begin{enumalpha}
		\item
			For any $m \geq 1$, we have $w_q(m) = 1$.
		\item
			For any $a \geq b \geq 1$ with $a+b=m$, we have $w_q(a,b) = q^{ab} + O_{p,a,b}(q^{ab-1})$.
		\item
			For any $a \geq 1$ with $2a+1=m$, we have $w_q(a,a,1) = q^{a(a+1)} + O_{p,a}(q^{a(a+1)-1})$.
	\end{enumalpha}
\end{lemma}

\begin{proof}
	Conjugating by an element of~$\GL_m(\F_p)$, we can assume without loss of generality that each~$V_k$ is generated by the first $\dim V_k = e(1) + \cdots + e(k)$ standard basis vectors of~$\F_p^m$.
	In particular, this proves the well-definedness of $w_q(e(1),\dots,e(s))$.
	\begin{enumalpha}
		\item
			That $e(1)=m$ implies that $V_1 = \bar\F_p^m$, and only the zero matrix satisfies $\ker M = V_1 = \bar\F_p^m$.
		\item
			The condition $\ker M^k = V_k$ for all $k\in\{1,2\}$ means that $M$ is of the form $M = \left(\begin{smallmatrix}0&N\\0&0\end{smallmatrix}\right)$ for some $a\times b$ matrix $N$ of rank $b$.
			If $M$ is of this form, then so is $\sigma(M)$ and they automatically commute.
			The number of such matrices $N$ with coefficients in $\F_q$ is $q^{ab} + O_{p,a,b}(q^{ab-1})$.
		\item
			The condition $\ker M^k = V_k$ for all $k\in\{1,2,3\}$ means that $M$ is of the form $M = \left(\begin{smallmatrix}0&N&\vec u\\0&0&\vec v\\0&0&0\end{smallmatrix}\right)$ for some invertible $a\times a$ matrix $N\in\GL_a(\F_q)$, some column vector $\vec u\in\F_q^a$, and some non-zero column vector $\vec v\in\F_q^a$.
			If $M$ is of this form, it commutes with $\sigma(M)$ if and only if $N\sigma(\vec v) = \sigma(N)\vec v$.
			Taking $\vec w := \sigma(\vec v)$, the claim then follows from \Cref{lem:count-matrices-vw} by summing over all possible pairs of vectors $\vec u\in\F_q^a$ and $\vec v \in \F_q^a \setminus \{0\}$, as $\vec v$ and $\sigma(\vec w) = \sigma^2(\vec v)$ are linearly independent if and only if $\langle\vec v\rangle$ is not defined over~$\F_{p^2}$, which is the generic case.
		\qedhere
	\end{enumalpha}
\end{proof}

\begin{theorem}
	\label{thm:eig}
	For any finite field $\F_q\supseteq\F_p$, we have
	\[
		|\A^\eig(\F_q)| = c^\eig(p,n)\cdot q^{\lfloor n^2/4\rfloor+1} + O_{p,n}(q^{\lfloor n^2/4\rfloor}),
	\]
	where
	\[
		c^\eig(p,2) = \tfrac12(p+2)(p+1),\qquad\qquad
		c^\eig(p,3) = \tfrac16(p^2+p+1)(p^4+7p^3+6p^2+6p+12),
	\]
	\[
		c^\eig(p,n) = \pbinom{n}{n/2} \quad\textnormal{if $n\geq4$ is even},
	\]
	\[
		c^\eig(p,n) = \pbinom{n}{\lfloor n/2\rfloor} + \pbinom{n}{\lfloor n/2\rfloor}\cdot\pbinom{\lceil n/2\rceil}{1} \quad\textnormal{if $n\geq5$ is odd}.
	\]
\end{theorem}

\begin{proof}
	For any Jordan shape $\Sh$ which is not listed in \Cref{prop:shapes-maximal-sameker}, we have $\dim\A^\eig_\Sh \leq \lfloor n^2/4\rfloor$ and therefore $|\A^\eig_\Sh(\F_q)| = O_{p,n}(q^{\lfloor n^2/4\rfloor})$ by the Schwartz--Zippel bound \cite[Lemma~1]{lang-weil}.

	Now, let $\Sh=(V,e)$ be one of the Jordan shapes listed in \Cref{prop:shapes-maximal-sameker}.
	To construct a matrix $M\in\A^\eig_\Sh$, we choose its $|V|$ (distinct) eigenvalues $\lambda_i$ and the corresponding generalized eigenspaces~$G_i$ of dimension $d(i):=\sum_{k\geq1}e(i,k)$ for all $i$ (which must be defined over~$\F_p$ by \Cref{cor:gen-eig-def}), modulo the automorphisms of~$\Sh$.
	There are $q^{|V|} + O_{p,n}(q^{|V|-1})$ choices for the eigenvalues and $|\GL_n(\F_p)|/\prod_{i\in V}|\GL_{d(i)}(\F_p)|$ choices for the generalized eigenspaces (as one shows using the orbit-stabilizer theorem).
	For each~$i$, we then need to choose the filtration of subspaces $V_{i,k} := \ker(M-\lambda_i I_n)^k$ (each defined over~$\F_p$), satisfying $0=V_{i,0}\subseteq\cdots\subseteq V_{i,s_i}=G_i$, with $\dim(V_{i,k}/V_{i,k-1}) = e(i,k)$.
	The group $\GL_{d(i)}(\F_p)$ acts transitively on such flags.
	Describing the stabilizer of a given flag (by induction on $s_i$) and using the orbit-stabilizer theorem, one shows that the number of such flags for each $i$ is
	\[
		\frac{
			|\GL_{d(i)}(\F_p)|
		}{
			\prod_{k\geq1} |\GL_{e(i,k)}(\F_p)|
			\cdot \prod_{k>l\geq1} p^{e(i,k)\cdot e(i,l)}
		}.
	\]
	Finally, we need to choose for each $i$ the restriction of~$M-\lambda_i I_n$ to the generalized eigenspace $G_i$.
	We estimated the number $w_q(e(i,1),e(i,2),\dots,e(i,s_i))$ of choices for this restriction in \Cref{lem:shapes-w}.
	For any Jordan shape $\Sh=(V,e)$, we then obtain
	\[
		|\A^\eig_\Sh(\F_q)|
		= \frac{|\GL_n(\F_p)|\cdot (q^{|V|} + O_{p,n}(q^{|V|-1}))}{|\Aut(\Sh)|}
		\cdot \prod_{i\in V} \frac{w_q(e(i,1),e(i,2),\dots,e(i,s_i))}{\prod_{k\geq1}|\GL_{e(i,k)}(\F_p)|\cdot\prod_{k>l\geq1}p^{e(i,k)\cdot e(i,l)}}.
	\]
	The claim follows by summing over all the shapes listed in \Cref{prop:shapes-maximal-sameker} and using the formulas given in \Cref{lem:shapes-w}.
\end{proof}

\section{Matrices commuting with their whole Frobenius orbit}
\label{sn:stable}

In this section, we determine the asymptotics of $|\A_\infty^\diag(\F_q)|$ and $|\A_\infty(\F_q)|$, i.e., we prove \Cref{thm:infty-diag,thm:infty-all}.
In \Cref{subsn:two}, we quickly deal with the case $n=2$ (in which case $\A(\F_q) = \A_\infty(\F_q)$).
In \Cref{subsn:decomp-stable}, we reduce the description of $\A_\infty^\diag(\F_q)$ and $\A_\infty(\F_q)$ to the enumeration of certain subalgebras of $\Mat_n(\F_p)$ (\Cref{lemma:stable}).
Then, \Cref{subsn:diagonalizable-stable} and \Cref{subsn:general-stable} deal with $|\A_\infty^\diag(\F_q)|$ and $|\A_\infty(\F_q)|$, respectively.

\subsection{The case of $2 \times 2$ matrices}
\label{subsn:two}

We first quickly deal with the case $n=2$, as it is particularly simple to obtain an exact count in this case, and the behavior is different compared to larger values of~$n$.

\begin{proposition}
	\label{prop:characterization-two-by-two}
	Assume that $n=2$, and let $M \in \Mat_2(\bar\F_p)$.
	The following are equivalent:
	\begin{enumroman}
		\item
			\label{prop:characterization-two-by-two-i}
			$M$ is of the form  $\lambda M' + \mu I_2$ with $\lambda, \mu \in \bar\F_p$ and $M' \in \Mat_2(\F_p)$.
		\item
			\label{prop:characterization-two-by-two-ii}
			$M \in \A_\infty$.
		\item
			\label{prop:characterization-two-by-two-iii}
			$M \in \A$.
	\end{enumroman}
\end{proposition}

\begin{proof}
	Clearly, \ref{prop:characterization-two-by-two-i}~$\Rightarrow$~\ref{prop:characterization-two-by-two-ii}~$\Rightarrow$~\ref{prop:characterization-two-by-two-iii}.
	Assume \ref{prop:characterization-two-by-two-iii}.
	If $M$ is a scalar matrix, \ref{prop:characterization-two-by-two-i} is clear.
	Otherwise, the condition that $M =
	\left(\begin{smallmatrix}
		a & b \\
		c & d
	\end{smallmatrix}\right)$ and $\sigma(M)=\left(\begin{smallmatrix}\sigma(a)&\sigma(b)\\\sigma(c)&\sigma(d)\end{smallmatrix}\right)$ commute rewrites as the following system of equations:
	\[
		\left\lbrace
		\arraycolsep=0.5ex
		\begin{array}{ccc}
			a \sigma(a) + b \sigma(c) & = & a \sigma(a) + c \sigma(b) \\
			a \sigma(b) + b \sigma(d) & = & b \sigma(a) + d \sigma(b) \\
			c \sigma(a) + d \sigma(c) & = & a \sigma(c) + c \sigma(d) \\
			c \sigma(b) + d \sigma(d) & = & b \sigma(c) + d \sigma(d) \\
		\end{array}
		\right.
		\qquad\Longleftrightarrow\qquad
		\left\lbrace
		\arraycolsep=0.5ex
		\begin{array}{ccc}
			b \sigma(c) & = & c \sigma(b) \\
			b\sigma(d-a) & = & (d-a)\sigma(b) \\
			c\sigma(d-a) & = & (d-a)\sigma(c)\mathrlap{,} \\
		\end{array}
		\right.
	\]
	meaning that the point $[b:c:d-a] \in \P^2(\bar\F_p)$ is $\sigma$-invariant, so belongs to~$\P^2(\F_p)$.
	Writing $(b,c,d-a) = \lambda(\beta,\gamma,\delta)$ with $\beta,\gamma,\delta \in \F_p$ and $\lambda \in \bar\F_p^\times$, we have~\ref{prop:characterization-two-by-two-i} with $\mu = a$ and $M' = \left(\begin{smallmatrix}
		0&\beta\\\gamma&\delta
	\end{smallmatrix}\right)$.
\end{proof}

\begin{corollary}
	\label{cor:count-2-by-2}
	Assume that $n=2$, and let $\F_q$ be a finite field containing $\F_p$.
	Then,
	$
		|\A(\F_q)|
		= |\A_\infty(\F_q)|
		= q + (p^2+p+1)(q-1)q.
	$
\end{corollary}

\begin{proof}
	Using \Cref{prop:characterization-two-by-two}, the size of $\A(\F_q) = \A_\infty(\F_q)$ is given by
	\[
		\begin{gathered}[b]
			\underbrace{
				q
			}_{\textnormal{scalar matrices}}
			\; + \;
			\underbrace{
				(p^2+p+1)
			}_{\textnormal{choices of } [b:c:d-a] \in \P^2(\F_p)}
			\cdot
			\underbrace{
				(q-1)
			}_{\substack{
				\textnormal{choices of } (b,c,d-a) \in \F_q^3 \setminus \{(0,0,0)\} \\
				\textnormal{once } [b:c:d-a] \textnormal{ is fixed}
			}}
			\cdot
			\underbrace{
				q
			}_{\textnormal{choices of } a}
			\\[-3\dp\strutbox]
		\end{gathered}
		\qedhere
	\]
\end{proof}

\subsection{The decomposition of $\A_\infty$}
\label{subsn:decomp-stable}

For any field~$K$, we call a subalgebra $A$ of~$\Mat_n(K)$ \emph{diagonalizable} if its elements are simultaneously diagonalizable over~$\bar K$.
In particular, a diagonalizable subalgebra is commutative.
The sets~$\A_\infty^\diag$ and~$\A_\infty$ can be decomposed using the (finitely many) diagonalizable (resp.\ commutative) subalgebras of $\Mat_n(\F_p)$:

\begin{lemma}
	\label{lemma:stable}
	We have
	\[
		\A_\infty^\diag =
		\bigcup_{\substack{A\subseteq\Mat_n(\F_p)\\\textnormal{diagonalizable}\\\textnormal{subalgebra}}} A\otimes_{\F_p}\bar\F_p
		\qquad\quad\textnormal{and}\qquad\quad
		\A_\infty =
		\bigcup_{\substack{A\subseteq\Mat_n(\F_p)\\\textnormal{commutative}\\\textnormal{subalgebra}}} A\otimes_{\F_p}\bar\F_p
		.
	\]
\end{lemma}

\begin{proof}
	The inclusions $\supseteq$ are clear: if $M\in A\otimes_{\F_p}\bar\F_p$ for a commutative subalgebra $A$ of~$\Mat_n(\F_p)$, then the matrices $\sigma^i(M)$ for $i = 0,1,...$ all belong to the commutative algebra $A\otimes_{\F_p}\bar\F_p$, hence commute with each other.
	If moreover $A$ is diagonalizable, then so is $M$.

	For the inclusions $\subseteq$, consider any matrix $M\in\A_\infty$.
	Since the matrices $M,\sigma(M),\dots$ commute, they generate a commutative $\bar\F_p$-subalgebra $R$ of~$\Mat_n(\bar\F_p)$.
	This subalgebra is $\sigma$-invariant, so by Galois descent we have $R=A\otimes_{\F_p}\bar\F_p$ for some commutative subalgebra $A$ of~$\Mat_n(\F_p)$, proving the second equality.
	If moreover $M\in\A_\infty^\diag$, then the commuting matrices $M,\sigma(M),\dots$ are diagonalizable, hence they are simultaneously diagonalizable.
	Any common eigenbasis of these matrices is in fact a common eigenbasis of all matrices in $R = A \otimes_{\F_p} \bar\F_p$, so $A \subseteq \Mat_n(\F_p)$ is a diagonalizable subalgebra.
\end{proof}

As a consequence of \Cref{lemma:stable} (using inclusion--exclusion), the size of~$\A_\infty^\diag(\F_q)$ (resp.~of~$\A_\infty (\F_q)$) is a polynomial in~$q$, whose degree and leading coefficient are given by the dimension and the number of the maximal-dimensional diagonalizable (resp.\ commutative) subalgebras of~$\Mat_n(\F_p)$.
We compute these numbers in \Cref{subsn:diagonalizable-stable} and \Cref{subsn:general-stable}, respectively.

\begin{remark}
	\Cref{lemma:stable} applies not just for matrices, but in any context where Galois descent holds.
	For example, if $\mathfrak g$ is a Lie $\F_p$-algebra
	, then a similar description holds for the elements of $\mathfrak g \otimes_{\F_p} \bar\F_p$
	commuting with their Frobenius orbit, reducing to the enumeration of the abelian Lie subalgebras of $\mathfrak g$ (cf.~\cite[Lemma~6.3]{wildcount} for an example).
	Similar statements hold if~$R$ is a $\F_p$-algebra and we want to describe the elements of $R \otimes \bar\F_p$ commuting with their Frobenius orbit (\Cref{lemma:stable} is the example $R = \Mat_n(\F_p)$), or if $\mathcal G$ is an algebraic group over $\F_p$ and we want to describe the elements of $\mathcal G(\bar\F_p)$ commuting with their Frobenius orbit, etc.
\end{remark}

\subsection{Diagonalizable matrices}
\label{subsn:diagonalizable-stable}

\begin{lemma}
	\label{lem:param-subalg-geomconj}
	~
  \begin{enumalpha}
	\item
		\label{item:param-subalg-geomconj-dim}
		Every diagonalizable subalgebra of~$\Mat_n(\bar\F_p)$ has dimension at most $n$.
	\item
		\label{item:param-subalg-geomconj-bij}
		There is a bijection between the set of $n$-dimensional diagonalizable subalgebras $A$ of~$\Mat_n(\bar\F_p)$ and the set of unordered $n$-tuples $\{E_1, \ldots, E_n\}$ of one-dimensional subspaces of~$\bar\F_p^n$ such that $E_1 \oplus \ldots \oplus E_n = \bar \F_p^n$.
	\item
		\label{item:param-subalg-geomconj-def}
		An $n$-dimensional diagonalizable subalgebra $A$ is defined over~$\F_p$ if and only if the corresponding unordered tuple $\{E_1, \ldots, E_n\}$ is $\sigma$-invariant, i.e., if there is a permutation $\pi \in \Sym_n$ such that $\sigma(E_i) = E_{\pi(i)}$.
	\item
		\label{item:param-subalg-geomconj-max-tori}
		There is a bijection between the set of $n$-dimensional diagonalizable subalgebras $A$ of~$\Mat_n(\bar\F_p)$ and the set of maximal tori in $\GL_n$.
		Moreover, $A$ is defined over~$\F_p$ if and only if the corresponding maximal torus is $\sigma$-invariant.
  \end{enumalpha}
\end{lemma}

\begin{proof}
	Let $A$ be any diagonalizable subalgebra of~$\Mat_n(\bar\F_p)$, and pick a common eigenbasis $\mathcal B = (e_1,\dots,e_n)$ of the matrices in $A$, so that
	every matrix in $A$ is diagonal when expressed in $\mathcal B$.
	We immediately obtain \ref{item:param-subalg-geomconj-dim}, and we see that if $A$ is $n$-dimensional, then it consists of all matrices which are diagonal with respect to $\mathcal B$.
	In this case, $\mathcal B$ is unique up to permutation and rescaling, as the spaces $\langle e_i \rangle$ are exactly the one-dimensional subspaces which are invariant under all matrices in $A$.
	Thus, $A\mapsto\{\langle e_1\rangle,\dots,\langle e_n\rangle\}$ defines a bijection as in \ref{item:param-subalg-geomconj-bij}.
	For \ref{item:param-subalg-geomconj-def}, note that if $e_1,\dots,e_n$ is a common eigenbasis of~$A$, then $\sigma(e_1),\dots,\sigma(e_n)$ is a common eigenbasis of~$\sigma(A)$.
	Combined with~\ref{item:param-subalg-geomconj-bij}, this implies that $A$ is fixed by $\sigma$ if and only if $\sigma$ permutes the eigenspaces $\langle e_1\rangle,\dots,\langle e_n\rangle$.
	Finally, for~\ref{item:param-subalg-geomconj-max-tori}, compare \cite[Observation~5.2]{repstab} to \ref{item:param-subalg-geomconj-bij} (the maximal torus associated to the subalgebra~$A$ is its unit group), and compare \cite[Remark under Definition~5.1]{repstab} to \ref{item:param-subalg-geomconj-def}.
\end{proof}

Let $c_\infty^\diag(p,n)$ be the number of $n$-dimensional diagonalizable subalgebras of~$\Mat_n(\F_p)$.

\begin{proposition}
	\label{thm:count-diag-algebras}
	We have $c_\infty^\diag(p,n) = p^{n^2-n}$.
\end{proposition}

By \iref{lem:param-subalg-geomconj}{item:param-subalg-geomconj-max-tori}, the number $c_\infty^\diag(p,n)$ is also the number of maximal tori of $\GL_n$ which are $\sigma$-invariant.
By \cite[14.16]{steinberg} (or more directly by \cite[Theorem~5.8]{repstab}), that number is~$q^{n^2-n}$.
Although this proof has the benefit of generalizing to arbitrary linear algebraic groups defined over~$\F_p$, we also include a standalone elementary proof for $\GL_n$:

\begin{proof}
	Distinguishing between the possible permutations $\pi$, and using the fact that $\Sym_n$ acts freely on ordered tuples of pairwise distinct spaces, \Cref{lem:param-subalg-geomconj} immediately implies:
	\begin{equation}
		\label{cor:mpn-as-sum}
		c_\infty^\diag(p,n)
		= \frac1{n!} \sum_{\pi\in \Sym_n} |N(\pi)|
	\end{equation}
	where~$N(\pi)$ is the set of \textbf{ordered} tuples $(E_1,\dots,E_n)$ of one-dimensional subspaces of $\bar\F_p^n$ such that $E_1 \oplus \ldots \oplus E_n = \bar\F_p^n$ and $\sigma(E_i) = E_{\pi(i)}$ for all $i=1,\dots,n$.

	We are now going to show that, for any permutation $\pi\in \Sym_n$, we have
	\begin{equation}
		\label{eqn:size-npi}
		|N(\pi)| = \frac{|\GL_n(\F_p)|}{\prod_{C\textnormal{ cycle in }\pi}(p^{|C|}-1)}.
	\end{equation}
	For this, we show that $\GL_n(\F_p)$ acts transitively on $N(\pi)$, with stabilizers isomorphic to $\prod_C \F_{p^{|C|}}^\times$.
	\Cref{eqn:size-npi} will then immediately follow using the orbit-stabilizer theorem.
	Let $C_1,\dots,C_r$ be the cycles of~$\pi$.
	For any $(E_1,\dots,E_n)\in N(\pi)$ and any cycle $C_k$ of~$\pi$, consider the subspace $F_k := \bigoplus_{i\in C_k}E_i$.
	Since $\pi$ permutes the elements of the cycle $C_k$, this subspace $F_k$ is by definition of~$N(\pi)$ fixed by $\sigma$ and hence defined over~$\F_p$.
	Moreover, $\bigoplus_k F_k = \bigoplus_i E_i = \bar\F_p^n$.
	The group $\GL_n(\F_p)$ acts transitively on the set of tuples $(F_1,\dots,F_r)$ of subspaces of~$\F_p^n$ such that $F_1 \oplus \ldots \oplus F_r = \F_p^n$ and $\dim F_k = |C_k|$ for all $k$, and the stabilizers for that action are isomorphic to~$\prod_k \GL(F_k)$.
	It is hence sufficient to prove, for fixed subspaces $F'_1,\ldots,F'_r$ of $\mathbb F_p^n$, that the action of~$\prod_k \GL(F'_k \otimes \bar\F_p)$ on the set of tuples $(E_1,\ldots,E_n)$ of one-dimensional subspaces of $\bar\F_p^n$ such that $\sigma(E_i) = E_{\pi(i)}$ and $\bigoplus_{i \in C_k} E_i = F'_k \otimes \bar\F_p$ is transitive, with stabilizers isomorphic to $\prod_k \F_{p^{|C_k|}}^\times$.
	As that action is ``block-diagonal'', we can restrict our attention to a single cycle.
	We now assume that $\pi = (1, \ldots, n)$, and we will show that we then have a ($\GL_n(\F_p)$-equivariant) bijection
	\[
	f \colon 
	\bigl\{
		\textnormal{$\F_p$-basis $(a_1,\dots,a_n)$ of~$\F_{p^n}$}
	\bigr\} / \F_{p^n}^\times
	\stackrel\sim\longrightarrow
	N(\pi)
	\]
	sending $[(a_1,\dots,a_n)]$ to the tuple $(E_1,\dots,E_n)$ where $E_1 = \langle(a_1,\dots,a_n)\rangle$ and $E_i = \sigma^{i-1}(E_1)$ for $i=2,\dots,n$.
	Since the group $\GL_n(\F_p)$ acts simply transitively on the set of $\F_p$-bases of~$\F_{p^n}$, it will then indeed act transitively on $N(\pi)$ with stabilizer isomorphic to $\F_{p^n}^\times$.
	It remains to show that the map $f$ is well-defined and bijective.
	For any $(E_1,\dots,E_n)\in N(\pi)$, we have $E_i=\sigma^{i-1}(E_1)$ for $i=2,\dots,n$ and $\sigma^n(E_1) = E_1$, so $E_1$ must be generated by a vector with coordinates in $\F_{p^n}$.
	Moreover, if we define $E_1 = \langle(a_1,\dots,a_n)\rangle$ and $E_i = \sigma^{i-1}(E_1)$ for $i=2,\dots,n$, then $E_1,\dots,E_n$ span~$\bar\F_p^n$ if and only if the matrix $\bigl(\sigma^{i-1}(a_j)\bigr)_{i,j}$ is invertible, which is equivalent to $a_1,\dots,a_n$ forming an $\F_p$-basis of~$\F_{p^n}$.%
	\footnote{
		If $\bigl(\sigma^{i-1}(a_j)\bigr)_{i,j}$ is singular, then there is a non-trivial linear combination $\sum_j \lambda_j \sigma^{i-1}(a_j) = 0$ with coefficients in~$\F_{p^n}$ between its columns, which amounts to $\sum_j \sigma^i (\lambda_j) a_j = 0$ for all $i \in \{0, \dots, n-1\}$, so the vector $(a_1, \ldots, a_n) \in (\F_{p^n})^n$ is orthogonal to the subspace $\Span_i \Bigl( \sigma^i(\lambda_1, \ldots, \lambda_n) \Bigr) \subseteq (\F_{p^n})^n$; that subspace is $\sigma$-invariant, hence admits an $\F_p$-basis, in particular it contains a non-zero vector in $\F_p^n$, which implies that there is a non-trivial linear combination $\sum_j \mu_j a_j = 0$ with coefficients in $\F_p$.
		Conversely, if $a_1, \ldots, a_n$ are linearly dependent over $\F_p$, then up to the action of~$\GL_n(\F_p)$, we can assume that $a_n = 0$ and then $\bigl(\sigma^{i-1}(a_j)\bigr)_{i,j}$ is singular as its last column vanishes.
	}
	We have thus shown \Cref{eqn:size-npi}.
	
	Now, for any partition of~$n$ with $n_\ell$ parts of size $\ell$, there are exactly $n!/\prod_{\ell\geq1}\ell^{n_\ell} \, n_\ell!$ permutations with~$n_\ell$ cycles of length $\ell$ (the centralizer of any such permutation is isomorphic to $\prod_\ell (\ZZ/\ell \ZZ)^{n_\ell} \rtimes \Sym_{n_\ell}$).
	Hence, \Cref{cor:mpn-as-sum} and \Cref{eqn:size-npi} imply
	\[
		\frac{
			c_\infty^\diag(p,n)
		}{
			|\GL_n(\F_p)|
		}
		=
		\sum_{\substack{
			\textnormal{partition of } n \\
			\textnormal{with } n_\ell \textnormal{ parts of size } \ell
		}}
			\frac
			{1}
			{
				\prod_{\ell \geq 1}
					\ell^{n_\ell} \,
					n_\ell ! \,
					(p^\ell - 1)^{n_\ell}
			}.
	\]
	As sizes of parts of partitions of~$n$ are characterized by the property $\sum_\ell \ell n_\ell = n$ (where $n_\ell \geq 0$ for all~$\ell$, and $n_\ell = 0$ for almost all $\ell$), the right-hand side is the coefficient in front of~$X^n$ of the power series
	{\allowdisplaybreaks
	\begin{align*}
		&
		\sum_{\substack{n_1, n_2, \ldots \geq 0\\\textnormal{almost all }0}}
			\,\,
			\prod_{\ell \geq 1}
				\frac
				{X^{\ell n_\ell}}
				{
					\ell^{n_\ell} \,
					n_\ell ! \,
					(p^\ell - 1)^{n_\ell}
				}
		=
		\prod_{\ell \geq 1}
			\,\,
			\sum_{n \geq 0}
				\frac
				{X^{\ell n}}
				{
					\ell^n \,
					n ! \,
					(p^\ell - 1)^n
				}
		=
		\prod_{\ell \geq 1}
			\exp\leftl(
				\frac{X^\ell}{\ell(p^\ell - 1)}
			\rightr)
		=
		\prod_{\ell \geq 1}
			\exp\leftl(
				\frac{p^{-\ell} X^\ell}{\ell(1 - p^{-\ell})}
			\rightr)
		\\
		={}&
		\exp\Biggl(
			\sum_{\substack{
				\ell \geq 1 \\
				k \geq 0
			}}
				\frac{p^{-\ell} X^\ell}{\ell} p^{- \ell k}
		\Biggr)
		=
		\exp\Biggl(
			-
			\sum_{k \geq 0}
				\ln(1-p^{-(1+k)} X)
		\Biggr)
		=
		\prod_{k \geq 1}
			\frac{1}{1-p^{-k} X}
		=
		\prod_{k \geq 1}
			\sum_{i \geq 0}
					p^{-ki} X^i
		\\
		={}&
		\sum_{n \geq 0}
			\Biggl(
				\sum_{\substack{
					i_1, i_2, \ldots \geq 0\\
					i_1 + i_2 + \ldots = n
				}}
					p^{- \sum_{k \geq 1} k i_k}
			\Biggr)
			X^n
		=
		\sum_{n \geq 0}
			\,\,
			\sum_{s\geq n}
				\Biggl|\Biggl\{
					i_1, i_2, \ldots \geq 0
					\;\Bigg\vert\;
					\begin{matrix}
						i_1 + i_2 + \ldots = n \\
						\sum_{k \geq 1} k i_k = s
					\end{matrix}
				\Biggr\}\Biggr|
				\cdot
				p^{-s}
			X^n.
	\end{align*}
	On the other hand:
	\begin{align*}
		& \sum_{n \geq 0}
			\frac{p^{n^2-n}}{|\GL_n(\F_p)|}
			X^n
		=
		\sum_{n \geq 0}
			\frac{
				p^{\frac{n(n-1)}2}
			}{
				(p^n-1)\cdots(p-1)
			}
			X^n
		=
		\sum_{n \geq 0}
			\Biggl(
				\prod_{k=1}^n
					\frac{
						p^{k-1}
					}{
						p^k-1
					}
			\Biggr)
			X^n
		=
		\sum_{n \geq 0}
			\frac1{p^n}
			\Biggl(
				\prod_{k=1}^n
					\frac{
						1
					}{
						1-p^{-k}
					}
			\Biggr)
			X^n
		\\
		={}&
		\sum_{n \geq 0}
			\frac1{p^n}
			\Biggl(
				\prod_{k=1}^n
					\,
					\sum_{i \geq 0}
						p^{-k i}
			\Biggr)
			X^n
		=
		\sum_{n \geq 0}
			\frac1{p^n}
			\Biggl(
				\sum_{i_1, \ldots, i_n \geq 0}
					\,\,
					\prod_{k=1}^n
						p^{-k i_k}
			\Biggr)
			X^n
		=
		\sum_{n \geq 0}
			\Biggl(
				\sum_{i_1, \ldots, i_n \geq 0}
					\,\,
					p^{-\left(\sum_{k=1}^n k i_k+n\right)}
			\Biggr)
			X^n
		\\
		={}&
		\sum_{n \geq 0}
			\,\,
			\sum_{s \geq n}
				\,\,
				\Biggl|\Biggl\{
					i_1, \ldots, i_n \geq 0
					\;\Bigg\vert\;
					\sum_{k = 1}^n k i_k = s-n
				\Biggr\}\Biggr|
				\cdot
				p^{-s}
				X^n.
	\end{align*}
	}
	Therefore, proving that $c_\infty^\diag(p,n) = p^{n^2-n}$ reduces to proving the following equality for all $s \geq n$:
	\[
		\Biggl|\Biggl\{
			i_1, i_2, \ldots \geq 0
			\;\Bigg\vert\;
			\begin{matrix}
				i_1 + i_2 + \ldots = n \\
				\sum_{k \geq 1} k i_k = s
			\end{matrix}
		\Biggr\}\Biggr|
		=
		\Biggl|\Biggl\{
			i_1, \ldots, i_n \geq 0
			\;\Bigg\vert\;
			\sum_{k = 1}^n k i_k = s-n
		\Biggr\}\Biggr|.
	\]
	We can interpret a list $(i_1, i_2, \ldots)$ such that $i_1 + i_2 + \ldots = n$ and $\sum_{k\geq1} k i_k = s$ as a partition of~$s$ with exactly $n$ (non-zero) parts ($i_k$ is the number of parts of size $k$).
	Similarly, we can interpret a tuple $(i_1, \ldots, i_n)$ such that $\sum_{k=1}^n k i_k = s - n$ as a partition of~$s-n$ whose parts all have size $\leq n$.

	Consider a partition of~$s$ with exactly $n$ parts.
	Removing $1$ from each part turns this partition into a partition of~$s-n$ with at most $n$ parts.
	Then, taking conjugate partitions turns that partition into a partition of~$s-n$ whose parts all have sizes~$\leq n$.
	As both of these operations can be inverted, we have described a bijection between the two sets, proving the claim.
\end{proof}

\begin{theorem}
	\label{thm:infty-diag}
	For any finite field $\F_q \supseteq \F_p$, we have
	\[
		|\A_\infty^\diag(\F_q)|
		=
		p^{n^2-n} \cdot q^n
		+
		O_{p,n}(q^{n-1}).
	\]
\end{theorem}

\begin{proof}
	By \Cref{lem:param-subalg-geomconj} and \Cref{thm:count-diag-algebras} (or \cite[Theorem~5.8]{repstab}), there are exactly $c_\infty^\diag(p,n) = p^{n^2-n}$ diagonalizable subalgebras of~$\Mat_n(\F_p)$ of dimension~$n$ and none of larger dimension.
	The claim thus follows from \Cref{lemma:stable} by inclusion-exclusion.
	(For any $n$-dimensional subalgebra $A$ of $\Mat_n(\bar\F_p)$ defined over~$\F_p$, we have $|A\cap\Mat_n(\F_q)| = q^n$, and for any two such subalgebras $A_1\neq A_2$, we have $|A_1\cap A_2\cap\Mat_n(\F_q)| \leq q^{n-1}$.)
\end{proof}

\subsection{General matrices}
\label{subsn:general-stable}

Let $n\geq3$.
We recall the definition of the Gaussian binomial coefficient
\[
	\pbinom{n}{k} := \frac{(p^n-1)\cdots(p^{n-k+1}-1)}{(p^k-1)\cdots(p-1)},
\]
which is the number of $k$-dimensional subspaces of~$\F_p^n$.

\begin{theorem}
	\label{thm:count-comm-algebras}
	The maximal dimension of a commutative subalgebra of~$\Mat_n(\F_p)$ is $\lfloor n^2/4 \rfloor+1$, and the number $c_\infty(p,n)$ of commutative subalgebras of that dimension is given by:
	\[
		c_\infty(p,3) = p^6 + p^5 + 3 p^4 + 3 p^3 + 3 p^2 + p + 1,
	\]
	\[
		c_\infty(p,n) =
		\pbinom{n}{n/2}
		\textnormal{ if $n\geq4$ is even},
		\qquad\qquad
		c_\infty(p,n) =
		2\pbinom{n}{\lfloor n/2 \rfloor}
		\textnormal{ if $n\geq5$ is odd.}
	\]
\end{theorem}

\begin{proof}
	The commutative subalgebras of maximal dimension of~$\Mat_n(\bar\F_p)$ were classified in \cite{schur} (see also \cite{mirza}).
	In particular, they have dimension $\lfloor n^2/4 \rfloor+1$.

	We now explain how to parametrize them.
	For any subspace $V \subsetneq \bar\F_p^n$, let $C_V$ be the linear subspace of matrices $A\in\Mat_n(\bar\F_p)$ such that $\im A\subseteq V\subseteq\ker A$, and let $C_V'$ be the algebra $C_V + \bar\F_p I_n$.
	The product of any two elements of~$C_V$ is zero, so $C_V'$ is a commutative subalgebra of~$\Mat_n(\bar\F_p)$.
	Moreover, $V$ can be recovered as the union of all images of nilpotent elements of~$C'_V$, so the map $V \mapsto C_V'$ is injective.
	We have $\sigma(C'_V) = C'_{\sigma(V)}$, so the algebra $C'_V$ is defined over~$\F_p$ if and only if~$V$ is defined over~$\F_p$.
	By \cite[Satz~II and Satz~III]{schur}, when $n>3$, the commutative subalgebras of~$\Mat_n(\bar\F_p)$ of (maximal) dimension $\lfloor n^2/4 \rfloor+1$ are exactly those of the form~$C_V'$ with $\dim V=\lfloor n/2\rfloor$ or $\dim V=\lceil n/2\rceil$.
	So, for $n > 3$, there are as many $(\lfloor n^2/4\rfloor+1)$-dimensional commutative subalgebras defined over~$\F_p$ as there are choices for such a subspace $V$ defined over~$\F_p$, namely $\pbinom{n}{n/2}$ for even $n$ and $\pbinom{n}{\lfloor n/2\rfloor} + \pbinom{n}{\lceil n/2\rceil} = 2\pbinom{n}{\lfloor n/2\rfloor}$ for odd~$n$.
	This proves the result for $n>3$.

	We now compute $c_\infty(p,3)$.
	According to \cite[Satz~II, Satz~III and p.~76]{schur}, there are five conjugacy classes (up to $\GL_3(\bar\F_p)$-conjugation) of three-dimensional commutative subalgebras of~$\Mat_3(\bar\F_p)$.
	In the following table, we list one representative $A$ of each conjugacy class and the number $N(A)$ of subalgebras defined over~$\F_p$ in the corresponding conjugacy class (the computations of $N(A)$ are detailed below the table):
	\begin{center}
		\renewcommand{\arraystretch}{2}
		\vspace{-1.5ex}
		\begin{tabular}{ccc}
			&representative $A$&$N(A)$\\\hline
			(1)&$\left\{\left(\begin{smallmatrix}\alpha&\beta&\gamma\\&\alpha\\&&\alpha\end{smallmatrix}\right) \;\middle\vert\; \alpha,\beta,\gamma\in\bar\F_p\right\}$
			&$p^2+p+1$\\
			(2)&$\left\{\left(\begin{smallmatrix}\alpha&&\beta\\&\alpha&\gamma\\&&\alpha\end{smallmatrix}\right) \;\middle\vert\; \alpha,\beta,\gamma\in\bar\F_p\right\}$
			&$p^2+p+1$\\
			(3)&$\left\{\left(\begin{smallmatrix}\alpha\\&\beta\\&&\gamma\end{smallmatrix}\right) \;\middle\vert\; \alpha,\beta,\gamma\in\bar\F_p\right\}$
			&$p^6$\\
			(4)&$\left\{\left(\begin{smallmatrix}\alpha&\beta\\&\alpha\\&&\gamma\end{smallmatrix}\right) \;\middle\vert\; \alpha,\beta,\gamma\in\bar\F_p\right\}$
			&$p^2(p^2+p+1)(p+1)$\\
			(5)&$\left\{\left(\begin{smallmatrix}\alpha&\beta&\gamma\\&\alpha&\beta\\&&\alpha\end{smallmatrix}\right) \;\middle\vert\; \alpha,\beta,\gamma\in\bar\F_p\right\}$
			&$(p^2+p+1)(p+1)(p-1)$
		\end{tabular}
	\end{center}
	Cases (1) and (2) correspond to the conjugacy classes $\{C'_V\mid V\subseteq\bar\F_p^3\textnormal{ one-dimensional}\}$ and $\{C'_V\mid V\subseteq\bar\F_p^3\textnormal{ two-dimensional}\}$, respectively, each of which contains $\pbinom{3}{1}=\pbinom{3}{2}=p^2+p+1$ subalgebras defined over~$\F_p$ (see the arguments above for odd $n>3$).
	Case~(3) corresponds to the conjugacy class of diagonalizable subalgebras, which according to \Cref{thm:count-diag-algebras} contains $p^6$ subalgebras defined over~$\F_p$.
	In cases (4) and (5), the $\GL_3(\bar\F_p)$-stabilizers~$S$ of~$A$ with respect to conjugation are respectively
	\[
		\left\{\left(\begin{smallmatrix}a&b\\&c\\&&d\end{smallmatrix}\right) \;\middle\vert\; a,c,d\in\bar\F_p^\times,\ b\in\bar\F_p\right\}
		\quad\textnormal{and}\quad
		\left\{
			\left(\begin{smallmatrix}a&b&c\\&d&e\\&&f\end{smallmatrix}\right)
			\;\middle\vert\;
			a,d,f\in\bar\F_p^\times,\ 
			b,c,e\in\bar\F_p,\ 
			\textnormal{with } af=d^2
		\right\}.
	\]
	In both cases, we have $H^1(\mathrm{Gal}(\bar\F_p|\F_p), S) = \{1\}$,%
	\footnote{
		By \cite[Chap.~X, \S 1, Exercise~2]{serre-local-fields}, the unit group of any algebra defined over~$\F_p$ has trivial first Galois cohomology.
		This directly shows case~(4), and case~(5) follows by looking at the long exact sequence in cohomology arising from the short exact sequence $1\to S\to T^\times \to\bar\F_p^\times\to1$, where~$T$ is the algebra of upper triangular matrices with coefficients in~$\bar\F_p$, and the homomorphism on the right is $\bigl(\begin{smallmatrix}a&b&c\\ &d&e\\ &&f\end{smallmatrix}\bigr) \mapsto afd^{-2}$.
	}
	so any algebra which is $\GL_3(\bar\F_p)$-conjugate to~$A$ and defined over~$\F_p$ is actually $\GL_3(\F_p)$-conjugate to $A$.%
	\footnote{
		If the algebra $U^{-1}AU$ is defined over~$\F_p$ for some $U\in\GL_3(\bar\F_p)$, we obtain a $1$-cocycle $\tau\mapsto U\tau(U)^{-1}\in S$.
		It must be a $1$-coboundary $\tau\mapsto T\tau(T)^{-1}$ for some $T\in S$, so $U':= T^{-1}U$ lies in $\GL_3(\F_p)$, and then $U^{-1}AU = U'^{-1}AU'$.
	}
	The size of the $\GL_3(\F_p)$-conjugacy class is $|\GL_3(\F_p)|/|S\cap\GL_3(\F_p)|$, which is verified to be the number given in the table.
	Summing everything, we find that $c_\infty(p,3) = p^6 + p^5 + 3 p^4 + 3 p^3 + 3 p^2 + p + 1$.
\end{proof}

As in the proof of \Cref{thm:infty-diag}, we deduce from \Cref{lemma:stable} and \Cref{thm:count-comm-algebras} the following theorem:

\begin{theorem}
	\label{thm:infty-all}
	Let $c_\infty(p,n)$ be as in \Cref{thm:count-comm-algebras}.
	For any finite field $\F_q \supseteq \F_p$, we have
	\[
		|\A_\infty (\F_q)|
		=
		c_\infty(p,n) \cdot q^{\lfloor n^2/4 \rfloor+1}
		+ O_{p,n}\leftl(
			q^{\lfloor n^2/4 \rfloor}
		\rightr).
	\]
\end{theorem}

\small
\emergencystretch=1em
\bibliographystyle{alphaurl}
\bibliography{refs.bib}

\end{document}